\documentclass[a4paper,12pt]{article}

\usepackage[utf8]{inputenc}
\usepackage{epsfig,graphics}
\usepackage{amssymb,amsmath,amsthm}
\usepackage{array,longtable,multirow}
\usepackage{authblk}
\usepackage[margin=30mm]{geometry}

\newtheorem{thm}{Theorem}
\newtheorem{lem}[thm]{Lemma}

\theoremstyle{definition}
\newtheorem{defn}[thm]{Definition}

\title{ Option Pricing under Multifactor Black--Scholes Model Using Orthogonal Spline Wavelets }

\date{}

\begin{document}

\begin{center}
    
{\LARGE Option Pricing under Multifactor Black--Scholes Model Using Orthogonal Spline Wavelets }

\vspace{5mm}

{\Large Dana \v{C}ern\'a\footnote{Department of Mathematics and Didactics of Mathematics, Technical University of Liberec, Studentská 2, Liberec, 461 17, Czech Republic}, 
Kate\v{r}ina Fi\v{n}kov\'a\footnote{Institute of Information Technology and Electronics, Technical University of Liberec, Studentská 2, Liberec, 461 17,  Czech Republic} }

\end{center}

\vspace{3mm}

\begin{abstract}
The paper focuses on pricing European-style options on several underlying assets under the Black--Scholes model represented by a nonstationary partial differential equation. The proposed method combines the Galerkin method with $L^2$-orthogonal sparse grid spline wavelets and the Crank--Nicolson sche\-me with Rannacher time-stepping. To this end, we construct an orthogonal cubic spline wavelet basis on the interval satisfying homogeneous Dirichlet boundary conditions and design a wavelet basis on the unit cube using the sparse tensor product. The method brings the following advantages. First, the number of basis functions is significantly smaller than for the full grid, which makes it possible to overcome the so-called ``curse of dimensionality." Second, some matrices involved in the computation are identity matrices, which significantly simplifies and streamlines the algorithm, especially in higher dimensions. Further, we prove that discretization matrices have uniformly bounded condition numbers, even without preconditioning, and that the condition numbers do not depend on the dimension of the problem. Due to the use of cubic spline wavelets, the method is higher-order convergent. Numerical experiments are presented for options on the geometric average.
\end{abstract}

\noindent {\it Keywords:} Black--Scholes model,  European option,  orthogonal spline wavelet, sparse grid, wavelet-Galerkin method, condition number 
\section{Introduction}

Methodology for the valuation of options has undergone rapid development in recent years;  nowadays, various option pricing models and numerical methods are available to determine the fair option price. This paper focuses on the famous Black--Scholes model \cite{Black1973}, which assumes that prices of underlying assets follow a log-normal distribution with constant drift and volatility. The Black--Scholes model is used for various types of options such as vanilla options, basket options, options on maximum or minimum of several assets, options on the geometric average, and real options.

Because the Black--Scholes model is a basic model for option pricing, the numerical valuation of options under this model is an active field of research, and there are a great many articles concerning this topic, such as the finite difference method studied in \cite{Lee2022}, higher-order spline methods designed in \cite{Kadalbajoo2012,Mohammadi2015,Roul2022}, and radial basis function approach examined in \cite{Milovanovic2020}. However, it is well known that when the Black--Scholes model is generalized to higher dimensions, e.g., dimension $d \geq 3$, the so-called ``curse of dimensionality" causes the number of degrees of freedom to grow exponentially with the number of dimensions. Therefore, the numerical solution of the model is more complicated, and far less literature has been devoted to this problem, see, e.g., \cite{Achdou2005,Hu2018,Leentvaar2008,Kim2016,Lyu2021}.
Wavelets have also been used for option pricing under the Black--Scholes model with one or more underlying assets, such as the B-spline wavelet-Galerkin method combined with discontinuous Galerkin time discretization \cite{Hilber2013}, an adaptive wavelet method \cite{Rometsch2010,Kestler2013}, the Legendre wavelet method \cite{Doostaki2022}, a method based on the wavelet approximation and the characteristic function \cite{Ortiz2013}, and the Shannon wavelet inverse Fourier technique~\cite{Ortiz2016}. 

In this paper, we deal with the Black--Scholes model for pricing multi-asset European options. For $d$ asset option, this model is represented by a nonstationary differential equation with $d$ state variables. First, we adjust the equation using the transformation into logarithmic prices, drift removal, approximation of the unbounded domain, and setting appropriate boundary conditions. 

Because the properties of Galerkin-type methods strongly depend on the wavelet basis used, we focus on constructing an appropriate one, namely an $L^2$-orthogonal cubic spline wavelet basis.
 In \cite{Donovan1996}, it was proved that if $\tilde{V}_j$ is a multiresolution analysis for the space $L^2 \left( \mathbb{R} \right)$, then there exists an orthogonal multiresolution analysis $V_j$ for the space $L^2 \left( \mathbb{R} \right)$ such that
\begin{equation} \label{intertwinning_MRA}
    \tilde{V}_q \subset V_0 \subset \tilde{V}_{q+n}
\end{equation}
for some $q \geq 0$ and $n \geq 1$, and construction principles of such a multiresolution analysis were introduced.
Based on \cite{Donovan1996}, orthogonal wavelets were constructed in \cite{Dijkema2009,Rupp2013}. However, the general method from \cite{Donovan1996} does not lead to uniquely determined scaling functions,
because it leads to the system of quadratic equations with more than one solution and because the part of the method is the orthogonalization of particular functions, which also does not lead to unique functions. In  \cite{Donovan1999}, the construction of $L^2$-orthogonal spline wavelets corresponding to the $L^2$-orthogonal spline functions from \cite{Donovan1996} was proposed. However, the resulting wavelet bases are not all equally suitable for use in numerical mathematics. While the $L^2$ condition number is one due to the $L^2$ orthogonality, the $H^1$ condition numbers differ for various wavelet bases. However, the $H^1$ condition number is also important because it influences the condition number of discretization matrices and the number of iterations needed to resolve the differential equation with the required accuracy. Furthermore, to be able to use the basis for the solution of the differential equation on a bounded domain, it is necessary to adapt the basis of $L^2 \left( \mathbb{R} \right)$ to a bounded interval. 

In this paper, we present a construction of orthogonal cubic scaling functions and wavelets on the real line. Then, we construct special boundary scaling functions and wavelets to adapt the basis from the real line to the interval. Finally, we use the sparse tensor product to create a wavelet basis. We show that when normalized in the $H^1$-seminorm, the resulting basis is well conditioned with respect to the $H^1$-seminorm, and study the properties of the basis. The coefficients of the constructed scaling functions and wavelets are presented in Appendix~A. 

Then, we propose and study a wavelet-based method for the multidimensional Black--Scholes model. The method combines the Galerkin method using sparse tensor product $L^2$-orthogonal spline wavelet bases and the Crank-Nicolson with Rannacher time-stepping. Due to the sparse tensor product structure and the orthogonality of the bases, the resulting scheme is greatly simplified compared to the standard Galerkin method as well as many other methods; it even overcomes the curse of dimensionality. The simplification is attributable to some of the matrices involved in computations being identity matrices. It is known that wavelet-based methods for certain types of equations lead to uniformly conditioned matrices when diagonal preconditioning is used \cite{Dahmen1992}. In this paper, we prove that the matrices are uniformly conditioned, even without preconditioning, and that the condition numbers can be bounded independently on the dimension of the problem. Because the matrices are symmetrical and positive definite, the resulting system can be solved using the conjugate gradient method. The uniform boundedness of condition numbers implies that the numbers of iterations are also bounded independently of the size of the matrix and the dimension of the problem. Another advantage is high-order convergence of the proposed method with constructed cubic spline-wavelet basis.

The paper is organized as follows. Section~\ref{Section_model} introduces the model and adjustment of the equation. Section~\ref{Section_construction} presents the assumption on the wavelet basis, construction of orthogonal cubic spline wavelet basis on the real line, its adaptation to a bounded interval and boundary conditions, and construction of wavelet basis on the unit cube using the sparse tensor product. The wavelet-based method is proposed and analyzed in Section~\ref{Section_method}. Numerical results are provided for options on the geometric average in Section~\ref{Section_examples}. 

\section{ Multifactor Black--Scholes model}
\label{Section_model}

The multi-dimensional Black--Scholes model is a generalization of the original Black--Scholes model for one-asset options designed in \cite{Black1973}. For an option on $d$ assets, the Black--Scholes model assumes that the price $S_i$ of the $i$-th underlying asset at time $\tilde{t}$ follows the diffusion process
\begin{equation} \label{diffusion_proces}
    \frac{ {\rm d} S_i }{ S_i } = \mu_i \, {\rm d} \tilde{t} + \sigma_i \, {\rm d} W_i, \quad i=1, \ldots, d;
\end{equation}
see also \cite{Achdou2005, Berridge2004, Hilber2013, Jo2013, Leentvaar2008}. The parameter $\mu_i$ represents the drift rate of the price $S_i$, and $W_i$ is the standard Brownian motion with volatility $\sigma_i$. The Wiener processes $W_i$ and $W_j$ are correlated with a correlation coefficient $\rho_{ij}$. Therefore, the corresponding covariance matrix $\mathbf{Q}$ is a $d \times d$ matrix with entries $\mathbf{Q}_{ij} = \rho_{ij} \sigma_i \sigma_j$. 

Using the stochastic equation (\ref{diffusion_proces}) and It\^{o} calculus, a deterministic equation for the expected value $V$ of the option is obtained. Let $r$ be a risk-free interest rate, $T$ be the maturity date, and $t= T - \tilde{t}$ be time to maturity.
Then, the fair value of the option $V \left(S_1, \ldots, S_d, t\right)$ corresponding to prices of underlying assets $S_i$ at the time to maturity $t$ can be computed as the solution of the partial differential equation
\begin{equation} \label{Black_Scholes}
\frac{ \partial V}{\partial t} =  \frac{1 }{2} \sum_{i=1}^d \sum_{j=1}^d \rho_{i,j} \sigma_i \sigma_j S_i S_j \frac{ \partial^2 V}{ \partial S_i \partial S_j}
   + r \sum_{i=1}^d S_i \frac{ \partial V}{\partial S_i } - r V,
\end{equation}
where $S_i > 0$ and $t \in \left( 0, T \right)$. Equation (\ref{Black_Scholes}) is degenerate in the sense that some second-order terms vanish when $S_i = 0$.

The equation has to be equipped with appropriate initial and boundary conditions, which depend on the option type. The initial condition represents the value of the option at maturity and has the form
\begin{equation}
V\left(S_1, \ldots, S_d,  0 \right)= V_0 \left( S_1, \ldots, S_d \right), \quad S_i > 0,
\end{equation}
where $V_0$ is the payoff function. Examples of payoff functions are given in Section~\ref{Section_examples}. Boundary conditions are discussed for transformed and localized equations further below.

\subsection{Transformation into log prices and drift re\-mo\-val}

Due to the degeneracy of equation (\ref{Black_Scholes}), it is not possible to derive a variational formulation and perform the analysis in the standard Sobolev spaces; special weighted Sobolev spaces must be used instead, see
\cite{Achdou2005}. Therefore, it is beneficial to remove this degeneracy first. It is well-known \cite{Achdou2005,Hilber2013} that this can be done using substitution into logarithmic prices $y_i = \log S_i$, $i=1, \ldots, d$. Then, the function 
\begin{equation}
U \left( y_1, \ldots, y_d, t\right) = V \left( e^{y_1}, \ldots, e^{y_d}, t \right) =   V \left( S_1, \ldots, S_d, t \right)
\end{equation}
representing the option value in logarithmic prices is the solution to the transformed equation 
\begin{equation} \label{eq_log_prices}
\frac{ \partial U}{\partial t} =  \frac{1 }{2} \sum_{i=1}^d \sum_{j=1}^d \rho_{i,j} \sigma_i \sigma_j \frac{ \partial^2 U}{ \partial y_i \partial y_j}
   + \sum_{i=1}^d \left( r - \frac{ \sigma_i^2 }{2} \right) \frac{ \partial U}{\partial y_i } - r U.         
\end{equation}
 Now, the differential operator on the right-hand side of equation (\ref{eq_log_prices}) is an elliptic operator with constant coefficients.

First-order terms in equation (\ref{eq_log_prices}) represent drift and  can be removed using another substitution, $x_i = y_i - b_i t$,
$b_i = \sigma_i^2 / 2 - r$, $i=1, \ldots, d.$ Then, a function
\begin{equation} 
    W \left( x_1, \ldots, x_d, t \right) = U \left( x_1 + b_1 t, \ldots, x_d + b_d t, t \right) = U \left( y_1, \ldots, y_d, t \right)
\end{equation}
is a solution to a symmetric problem
\begin{equation} \label{eq_BS_symmetric_1}
 \frac{ \partial W}{\partial t} = \frac{1 }{2} \sum_{i=1}^d \sum_{j=1}^d \rho_{i,j} \sigma_i \sigma_j \frac{ \partial^2 W}{ \partial x_i \partial x_j} - r W.
\end{equation}

\subsection{Localization}

In order to solve equation (\ref{eq_BS_symmetric_1}) numerically, it is convenient to approximate the unbounded domain $\mathbb{R}^d$ for $\left( x_1, \ldots, x_d \right)$ with a bounded domain.
Therefore, we choose minimal and maximal values for $S_i$, denote them $S_i^{min}$ and $S_i^{max}$, respectively, and define
\begin{equation}
x_i^{min} = \ln S_i^{min}, \quad x_i^{max} = \ln S_i^{max}, \quad 
I_i = \left( x_i^{min}, x_i^{max} \right).
\end{equation}

Then, the equation to be solved is
\begin{equation} \label{eq_bounded_domain_1}
 \frac{ \partial W}{\partial t} - \frac{1 }{2} \sum_{i=1}^d \sum_{j=1}^d \rho_{i,j} \sigma_i \sigma_j \frac{ \partial^2 W}{ \partial x_i \partial x_j} + r W = 0, \quad x_i \in I_i, \quad t \in \left( 0, T \right).
\end{equation}

In Section~\ref{Section_construction}, a wavelet basis $\Psi$ on
$\Omega = \left( 0, 1 \right)^d$ is designed. To solve (\ref{eq_bounded_domain_1}), we can either use a linear transformation to convert $\Psi$ to the product domain $I_1 \times I_2 \times \ldots \times I_d$ or transform (\ref{eq_bounded_domain_1}) to $\Omega$. Here, we use the second approach. Let $d_i = x_i^{max} - x_i^{min}$ and
\begin{equation} 
    u \left( x_1, \ldots, x_d, t \right) = W \left( \frac{x_1 - x_1^{min}}{d_1}, \ldots, \frac{ x_d - x_d^{min} }{d_d} \right).
\end{equation}
The function $u$ is a solution to 
\begin{equation} \label{eq_bounded_domain_2}
 \frac{ \partial u}{\partial t} - \mathcal{ D } \left( u \right) = 0, 
 \quad \left( x_1, \ldots, x_d \right) \in \Omega, \quad t \in \left( 0, T \right),
\end{equation}
where
\begin{equation} \label{definition_operator_D}
    \mathcal{ D } \left( u \right) = \sum_{i=1}^d \sum_{j=1}^d P_{i,j}  \frac{ \partial^2 u}{ \partial x_i \partial x_j} - r u, \quad
    P_{i,j} = \frac{ \rho_{i,j} \sigma_i \sigma_j }{ 2 \, d_i d_j}.
\end{equation}
The initial condition is
\begin{equation} \label{init_condition}
    u \left(  x_1, \ldots, x_d, 0 \right) = u_0 \left( x_1, \ldots, x_d \right), \quad  \left( x_1, \ldots, x_d \right) \in \Omega,
\end{equation}
where $u_0$ is the transformed payoff function, 
\begin{equation} 
    u_0 \left(  x_1, \ldots, x_d \right) = V_0 \left( \exp \left( \frac{x_1 - x_1^{min} }{ d_1 }  \right), \ldots, \exp \left(  \frac{ x_d - x_d^{min} }{ d_d } \right) \right).
\end{equation}

Prescribing appropriate boundary conditions is a crucial task because exact values of $u$ on the boundary $\partial \Omega$ of a domain $\Omega$ are generally not known, and the computation of approximate values on the boundary can be quite complicated. Therefore, it is convenient to approximate the boundary conditions by homogeneous Dirichlet boundary conditions
\begin{equation} \label{boundary_conditions}
   u \left(  x_1, \ldots, x_d, t \right) = 0, \quad  \left( x_1, \ldots, x_d \right) \in \partial \Omega, \quad t \in \left( 0, T \right).
\end{equation}
However, because these boundary conditions are artificial and are not satisfied by the original function $u$, the error can be substantial in the vicinity of $\partial \Omega$. Thus, to obtain a sufficiently accurate solution in a chosen region of interest in $\Omega$, it is essential to choose sufficiently large $\Omega$, see \cite{Hilber2013}.

\section{ Orthogonal Spline Wavelets }
\label{Section_construction}

The efficiency of Galerkin-type methods depends crucially on the used basis functions. As already mentioned, this paper focuses on
the wavelet-Galerkin method, and thus, the appropriate wavelet basis is needed. In this section, we make assumptions on the wavelet basis on the bounded interval, and we give a detailed description of a construction of a wavelet basis meeting all the assumptions, specifically an orthogonal cubic spline wavelet basis. First, orthogonal cubic scaling functions and wavelets on the real line are designed. Special boundary scaling functions and wavelets are then proposed to adapt the basis from the real line to the bounded interval and to the homogeneous Dirichlet boundary conditions. Finally, the sparse tensor product is used to construct a wavelet basis on the product domain $\Omega$. We show that the resulting basis is well conditioned with respect to the $H^1$-seminorm and study its properties.

Let $\left\langle \cdot, \cdot \right\rangle$ and $\left\| \cdot \right\|$ denote the inner product and the norm of the space $L^2 \left( \Omega \right)$, respectively. Let $\left\| \cdot \right\|_s$ be the norm of the Sobolev space $H^s \left( \Omega \right)$, and $\left| \cdot \right|_1$ be the seminorm in $H^1 \left( \Omega \right)$. The space of functions from $H^1 \left( \Omega \right)$ with zero trace on $\partial \Omega$ is denoted $H_0^1 \left( \Omega \right)$. 

 Furthermore, let $\mathcal{J}$ be an index set such that each index $\lambda \in \mathcal{J}$ takes the form $\lambda = \left(j,k\right)$, and let $\left| \lambda \right|=j$ represent a level. For $\mathbf{v}=\left\{ v_{\lambda} \right\}_{\lambda \in \mathcal{J}}$, $v_{\lambda} \in \mathbb{R}$, define
\begin{equation}
\left\| \mathbf{v} \right\|_2= \left( \sum_{\lambda \in \mathcal{J}} 
 \left| v_{\lambda} \right|^2 \right)^{1/2}, \quad
l^2 \left( \mathcal{J} \right)=\left\{ \mathbf{v}: \left\| \mathbf{v} \right\|_2 < \infty \right\}.
\end{equation}

We make assumptions on a family of functions $\Psi=\left\{  \psi_{\lambda}, \lambda \in \mathcal{J} \right\}$ to be a suitable orthogonal spline-wavelet basis for 
$L^2 \left( I \right)$, where $I$ is a bounded interval or $I = \mathbb{R}$.

\begin{itemize}
\item[$A1)$] {\bf $L^2$ orthogonality.} The set $\Psi$ is an orthogonal basis of
$L^2 \left( I \right)$, and basis functions are normalized, i.e., $\left\langle \psi_{\lambda}, \psi_{\lambda} \right\rangle = 1$ and $\left\langle \psi_{\lambda}, \psi_{\mu} \right\rangle = 0,$ if $\lambda \neq \mu.$ 

\item[$A2)$] {\bf Locality.}
 For all $\lambda \in \mathcal{J}$, a diameter of the support of $\psi_{\lambda}$ is bounded by  $C 2^{-\left| \lambda \right|}$ with $C$ independent of $\left| \lambda \right|$ and at a given level, the supports of only finitely many wavelets
overlap at any point $x$. 

\item[$A3)$] {\bf Spline basis.} Basis functions are continuous  piecewise-polynomial functions of order $p \geq 2$. 

\item[$A4)$] {\bf Hierarchical structure.} 
The set $\Psi$ has a hierarchical structure
\begin{equation} \label{definition_Psi}
\Psi=\Phi_{j_0} \cup \bigcup_{j=j_0}^{\infty} \Psi_j,
\end{equation}
where
\begin{equation} 
\Phi_{j_0}=\left\{ \phi_{j_0,k}, k \in \mathcal{I}_{j_0} \right\}, \quad \Psi_{j}=\left\{ \psi_{j,k}, k \in \mathcal{J}_j\right\}.
\end{equation}
Functions $\phi_{j,k}$ are called {\it scaling functions}, and functions $\psi_{j,k}$ are called {\it wavelets}. Furthermore, assume that all the wavelets $\psi_{j,k}$ are translations and dilations of certain wavelet generators $\psi_l$,
i.e., they are of the form
\begin{equation} \label{psi_jk_translations_dilations}
\psi_{j,k} = 2^{j/2} \psi_l \left( 2^j \cdot - m \right),    
\end{equation}
where $l$ and $m$ are integers dependent on $k$.

\item[$A5)$] {\bf Vanishing moments.} 
We assume that wavelets~$\psi_{j,l}$ have $L \geq 1$ vanishing moments, i.e.,
for any polynomial $p$ of degree less than $L$, 
\begin{equation} \label{vanishing_moments}
\int_{ I }
 p \left( x \right) \, \psi_{j,l} \left(x\right) dx =0, \quad l \in \mathcal{J}_j, \ j \geq j_0.
\end{equation}
In some constructions, condition~(\ref{vanishing_moments}) cannot be satisfied by so-called boundary wavelets. In these cases, we admit an exception and we require that (\ref{vanishing_moments}) be satisfied only by inner wavelets.
\end{itemize}

Furthermore, we need to work with spaces $H_0^1 \left( I \right)$, where $I$ is a bounded interval. First, recall a definition of a Riesz basis.

\begin{defn}
Let $H$ be a separable Hilbert space with the norm $\left\| \cdot \right\|_H$. The set $\Gamma = \left\{ \gamma_{\lambda}, \lambda \in \mathcal{J} \right\}$ is called a {\it Riesz basis} of $H$ if
the closure of the span of $\Gamma$ is $H$ and there exist constants $c,C \in \left( 0, \infty \right)$ such that for all $\mathbf{b} = \left\{ b_{\lambda}\right\}_{\lambda \in \mathcal{J}} \in l^2 \left( \mathcal{J} \right)$, we have
\begin{equation} \label{riesz}
c \left\| \mathbf{ b} \right\|_2 \leq \left\| \sum_{\lambda \in \mathcal{J}} b_{\lambda} \gamma_{\lambda} \right\|_H \leq C \left\| \mathbf{b} \right\|_2. 
\end{equation}
Constant $c_{\psi}:=\text{sup} \left\{c: c \ \text{satisfies} \ (\ref{riesz}) \right\}$ is called {\it a lower Riesz bound}, constant $C_{\psi}:=\text{inf} \left\{C: C \ \text{satisfies} \ (\ref{riesz}) \right\}$ is called {\it an upper Riesz bound}, and $cond \ \Psi = C_{\psi}/c_{\psi}$ is called {\it a condition number} of $\Psi$.

If $\Gamma$ satisfies (\ref{riesz}), but the closure of its span is not necessarily dense in $H$, then $\Gamma$ is called a {\it Riesz sequence} in $H$. 
\end{defn}

Now, we require that functions from $\Psi$ normalized in the $H^1$-norm be a Riesz basis of the space $H_0^1 \left( I \right)$. Thus, we formulate  another assumption $A6)$.

\begin{itemize}
\item[$A6)$] { \bf $H^1$ Riesz basis.} The set $\left\{ \psi_{\lambda} / \left\| \psi_{\lambda} \right\|_{1}, \psi_{\lambda} \in \Psi \right\}$ is a Riesz basis of the space $H_0^1 \left( I \right)$. 
\end{itemize}

The construction of wavelets typically starts with a multiresolution analysis of the space $L^2 \left( \mathbb{R} \right)$.

\begin{defn}
The sequence $\left\{ V_j \right\}_{ j \in \mathbb{Z} }$ of closed subspaces in $L^2 \left( \mathbb{R} \right)$ is called {\it a multiresolution analysis (MRA)} of multiplicity $p$, if
the following conditions are satisfied:
\begin{itemize}
    \item[i)] $V_j \subset V_{j+1}$ for all $j \in \mathbb{Z}$.
    \item[ii)] For all $j \in \mathbb{Z}$, $f \in V_j $ if and only if $f \left( 2^{-j} \cdot \right) \in V_0$.
    \item[iii)] $ \bigcup\limits_{ j \in \mathbb{Z} } V_j $ is dense in $L^2 \left( \mathbb{R} \right)$.
    \item[iv)] $ \bigcap\limits_{ j \in \mathbb{Z} } V_j = \left\{ 0 \right\}$.
    \item[v)] There exist functions $\phi_1, \ldots, \phi_p$ such that
    \begin{equation}
       \Phi_0 = \left\{ \phi_i \left( \cdot - k \right), i=1, \ldots, p, k \in \mathbb{Z} \right\} 
    \end{equation}
    is a Riesz basis of $V_0$.
\end{itemize}
Functions $\phi_i$, $i=1, \ldots, p,$ are said to {\it generate} the multiresolution analysis $\left\{ V_j \right\}_{j \in \mathbb{Z}}$ and are called {\it scaling generators}. If $\Phi_0$ is an orthonormal basis of~$V_0$, then $\left\{ V_j \right\}_{j \in \mathbb{Z}}$ is called {\it an orthogonal multiresolution analysis}.
\end{defn}

\subsection{ Construction of orthogonal cubic spline scaling basis }

In this section, we use some of the general design principles from \cite{Donovan1996} to create a scaling basis. We start with the well-known scaling generators already used in \cite{Cerna2016, Dahmen2000}. Let $\xi_1$ and $\xi_2$ be Hermite cubic splines defined as
\begin{equation}
\xi_1(x) = \left\{ \begin{array}{cl}
\left(x+1\right)^2 \left( 1 - 2x \right), & x \in [-1,0], \\
\left(1-x\right)^2 \left( 1 + 2x \right), & x \in [0,1], \\
0, & {\rm otherwise}, 
\end{array} \right.
\end{equation}
and
\begin{equation}
\xi_2(x) = \left\{ \begin{array}{cl}
\left(x+1\right)^2 x, & x \in [-1,0], \\
\left(1-x\right)^2 x, & x \in [0,1], \\
0, & {\rm otherwise}. 
\end{array} \right. 
\end{equation}
Figure~\ref{Fig1} displays graphs of functions $\xi_1$ and $\xi_2$. 

\begin{figure}[ht!]
\centering
\includegraphics[width=13.0cm]{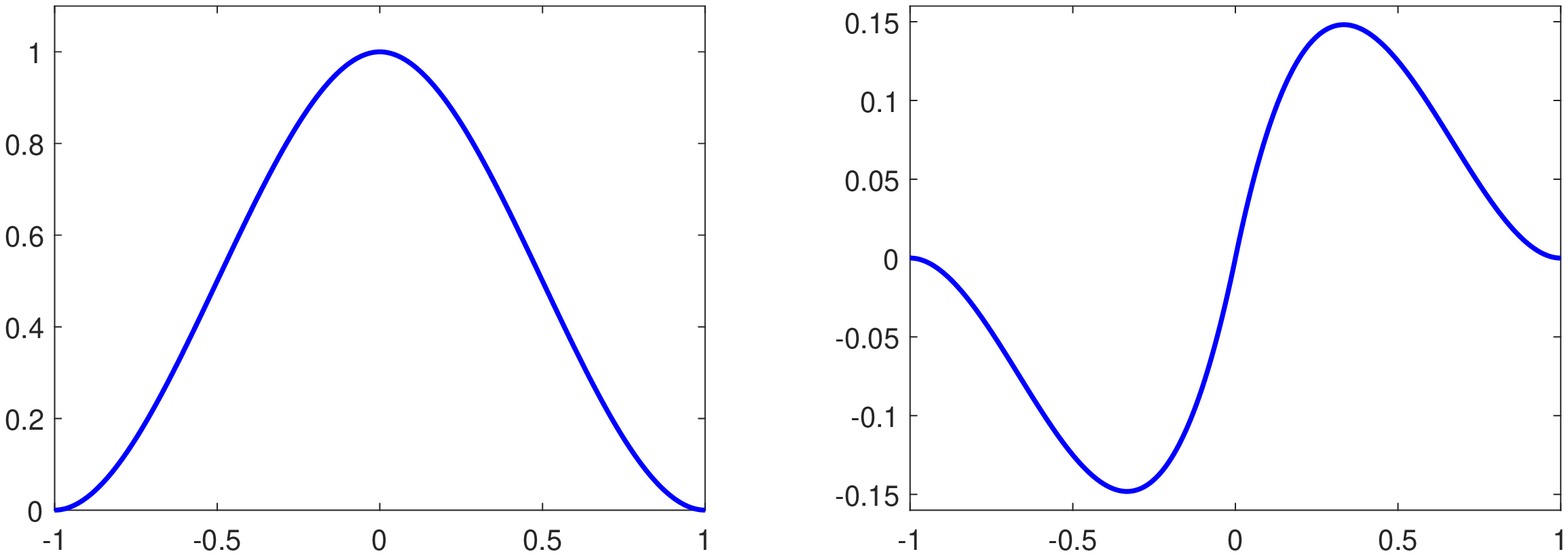} 
\caption{ Hermite cubic spline generators $\xi_1$ (left) and $\xi_2$ (right). }
\label{Fig1}
\end{figure} 

For $j \in \mathbb{N}$ and $k \in \mathbb{Z}$, let $\xi_{j,k}$ be translations and dilations of these two generators, 
\begin{equation} \label{definition_scaling_functions}
\xi_{j,2k+l-1} \left( x \right) = 2^{j/2} \xi_l \left( 2^j x -k \right) \ {\rm for} \quad l=1,2.
\end{equation} 
It is well known that spaces 
\begin{equation} \label{space_tilde_Vj}
\tilde{V}_j = \overline{ {\rm span } \left\{ \xi_{j,k}, k \in \mathbb{Z} \right\}  }
\end{equation}
form a multiresolution analysis, which is not orthogonal; see \cite{Cerna2016,Dahmen2000}. The aim of this section is to construct an orthogonal MRA formed by spaces $V_j$ satisfying (\ref{intertwinning_MRA}). This MRA cannot be generated by only two functions $\xi_1$ and $\xi_2$ as in (\ref{space_tilde_Vj}), but several new generators must be constructed.

As in \cite{Donovan1996}, these generators are constructed to have their supports in  $\left[ -1, 1 \right]$. Because it is required that (\ref{intertwinning_MRA}) be satisfied, these generators must be linear combinations of functions $\xi_{j,k}$ supported in $\left[ -1, 1 \right]$ for a certain level~$j$. For $j \leq 0$, there are not enough functions $\xi_{j,k}$ supported in $\left[ -1, 1 \right]$; therefore we start with the level $j=1$.

First, we limit ourselves to the interval $I_0 = \left[ 0, 1 \right]$. For $j=1$, there are two functions $\xi_{1,2}$ and $\xi_{1,3}$ supported in  $I_0$. Due to the symmetry of $\xi_1$ and the antisymmetry of $\xi_2$, it is easy to verify that functions
\begin{equation}
    \phi_1 = \frac{ \xi_{1,2} }{ \left\| \xi_{1,2} \right\| }, \quad
    \phi_2 = \frac{ \xi_{1,3} }{ \left\| \xi_{1,3} \right\| }
\end{equation}
are $L^2$-orthogonal. The scaling generators $\phi_1$ and $\phi_2$ are displayed in Figure~\ref{Fig2}. 

To construct another two generators with supports in $I_0$, we take functions from $\tilde{V}_2$ with supports in $I_0$, which are functions $\xi_{2,k}$ for $k=2, \ldots, 7$. Let define spaces
\begin{equation}
\mathcal{A}_0 = \text{span} \left\{ \xi_{1,2}, \xi_{1,3} \right\}, \quad
\mathcal{A}_1 = \text{span} \left\{ \xi_{2,k}, k=2, \ldots, 7 \right\}.  
\end{equation}
This means that $\mathcal{A}_0$ is the space generated by all functions $\xi_{j,k}$ on the coarsest level $j=1$ supported in $I_0$ and $\mathcal{A}_1$ is the space generated by all functions $\xi_{j,k}$ on level $j=2$ supported in $I_0$. Next, we find an orthogonal basis of $\mathcal{A}_1 \ominus \mathcal{A}_0$, where $\ominus$ denotes the orthogonal complement. 

Let $\mathbf{C}$ be the matrix with entries
\begin{equation}
    C_{i,j} = \left< \xi_{1,i+1}, \xi_{2,j+1} \right>, \quad i=1,2, \quad j=1, \ldots, 6,
\end{equation}
and $\mathbf{b}^1, \ldots, \mathbf{b}^4$ be vectors forming a basis of the null space of $\mathbf{C}$. Defining functions 
\begin{equation}
    w_i = \sum_{j=1}^6 b^i_{j} \, \xi_{2,j+1}, \quad i = 1, \ldots, 4,
\end{equation}
and orthogonalizing them, we obtain functions
\begin{equation}
 v_1 = \frac{ w_1 }{ \left\| w_1 \right\|}, \quad
    v_i = \frac{ w^i - \sum_{j=1}^{i-1} \left< w_i, v_j \right> v_j }{ \left\| w_i - \sum_{j=1}^{i-1} \left< w_i, v_j \right> v_j  \right\|}, \quad i=2,3,4,   
\end{equation}
which are normalized and form an orthogonal basis for $\mathcal{A}_1 \ominus \mathcal{A}_0$.

The next step of the construction is based on the theory developed in \cite{Donovan1996}. For the convenience of the reader, similar notations for spaces are used as therein. Let
\begin{eqnarray} \label{def_Bh}
\mathcal{B}_0 &=& \text{span} \left( \left\{ \xi_{1,0} \chi_{[0,1]}, \xi_{1,1} \chi_{[0,1]} \right\} \cup \mathcal{A}_0
\right), \\ 
\nonumber \mathcal{B}_1 &=& \text{span} \left( \left\{ \xi_{1,4} \chi_{[0,1]}, \xi_{1,5} \chi_{[0,1]} \right\}  \cup \mathcal{A}_0
\right),    
\end{eqnarray}
where $\chi_{[0,1]}$ is the characteristic function of the interval $\left[0,1 \right]$. This means that $\mathcal{B}_0$ is the space generated by left boundary functions and $\mathcal{B}_1$ is the space generated by right boundary functions. Furthermore, for $h=0,1,$ let 
 \begin{equation} \label{def_Ch}
 \mathcal{C}_h =  \mathcal{B}_h \ominus \mathcal{A}_0. 
 \end{equation}

The following lemma is crucial for the construction. 

\begin{lem} \label{condition_DGH}
With the above settings, let $W$ be a subspace of $\mathcal{A}_1 \ominus \mathcal{A}_0$ such that
\begin{equation} \label{cond_DGH}
 \left(  I - P_W \right) \mathcal{C}_0 \perp  \left(  I - P_W \right) \mathcal{C}_1,
\end{equation}
where $P_W$ is the orthogonal projection, $I$ is identity, and
$\perp$ denotes orthogonality.
If $\phi_3$ and $\phi_4$ form an orthogonal basis for $W$, then $V_j$ generated by $\phi_1$, $\phi_2$, $\phi_3$, $\phi_4$, $\xi_{1,0}$ and $\xi_{1,1}$ form an $L^2$-orthogonal multiresolution analysis satisfying~(\ref{intertwinning_MRA}).
\end{lem}

\begin{proof}
Lemma \ref{condition_DGH} is a consequence of a more general Lemma 3.1 from~\cite{Donovan1996}.
\end{proof}

It follows from (\ref{def_Bh}) and (\ref{def_Ch}) that the space $\mathcal{C}_0$ is generated by functions
\begin{equation}
   c_i = \xi_{1,i-1} \chi_{[0,1]} - \left< \xi_{1,i-1}, \phi_1 \right> \phi_1 -
    \left< \xi_{1, i-1}, \phi_2 \right> \phi_2, \quad i=1,2. 
\end{equation}
Similarly, the space $\mathcal{C}_1$ is generated by functions
\begin{equation}
 c_i = \xi_{1,i+1} \chi_{[0,1]} - \left< \xi_{1,i+1}, \phi_1 \right> \phi_1 -
  \left< \xi_{1,i+1}, \phi_2 \right> \phi_2, \quad i=3,4.    
\end{equation}
Because $W \subset \mathcal{A}_1 \ominus \mathcal{A}_0$ and 
$v_1, v_2, v_3,$ and $v_4$ form a basis for $\mathcal{A}_1 \ominus \mathcal{A}_0$, basis functions $\phi_3$ and $\phi_4$ for $W$ can be expressed as 
\begin{equation}
    \phi_3 = \sum_{j=1}^{4} s_{1,j} v_j, \quad \phi_4 = \sum_{j=1}^{4} s_{2,j} v_j,
\end{equation}
where the coefficients $s_{i,j}$ are such that condition (\ref{cond_DGH}) is satisfied.
This condition can be rewritten as
\begin{equation}
0 = \left< c_i - \sum_{m=1}^2 \sum_{j=1}^{4} s_{m,j} \left< c_i, v_j \right> v_j,
c_k - \sum_{n=1}^2 \sum_{l=1}^{4} s_{n,l} \left< c_k, v_l \right> v_l \right>,
\end{equation}
for $i=1,2$, and $k=3,4.$ Due to the orthogonality of functions $v_j$, we have
\begin{equation} \label{kvadr_cond}
    \left< c_i, c_k \right> = \sum_{m=1}^2 \sum_{j=1}^{4} s_{i,j} s_{k,j} \left< c_i, v_j \right>
    \left< c_k, v_j \right>.
\end{equation}
Furthermore, we require $\phi_3$ and $\phi_4$ to be orthogonal and normalized, which yields additional conditions
\begin{equation} \label{ortog_cond}
    \sum_{j=1}^4 s_{i,j}^2 = 1, \quad i=1,2, \quad \sum_{j=1}^4 s_{1,j} s_{2,j} = 0.
\end{equation}
Hence, (\ref{kvadr_cond}) and (\ref{ortog_cond}) form a system of $7$ quadratic equations with $8$ unknowns, which cannot be solved explicitly. Therefore, we solve it numerically by the Levenberg--Marquardt algorithm \cite{Marquardt1963}. According to Lemma~\ref{condition_DGH}, the resulting functions $\phi_3$ and $\phi_4$ are other two scaling generators orthogonal to $\phi_1$ and $\phi_2$. All the scaling generators supported in $I_0$ are displayed in Figure~\ref{Fig2}. 

\begin{figure}[ht]
\caption{\emph{ Orthogonal cubic spline generators $\phi_1$ (top left), $\phi_2$ (top right), $\phi_3$ (bottom left), and $\phi_4$ (bottom right) supported in $\left[ 0, 1 \right]$. }}
\includegraphics[width=13.0cm]{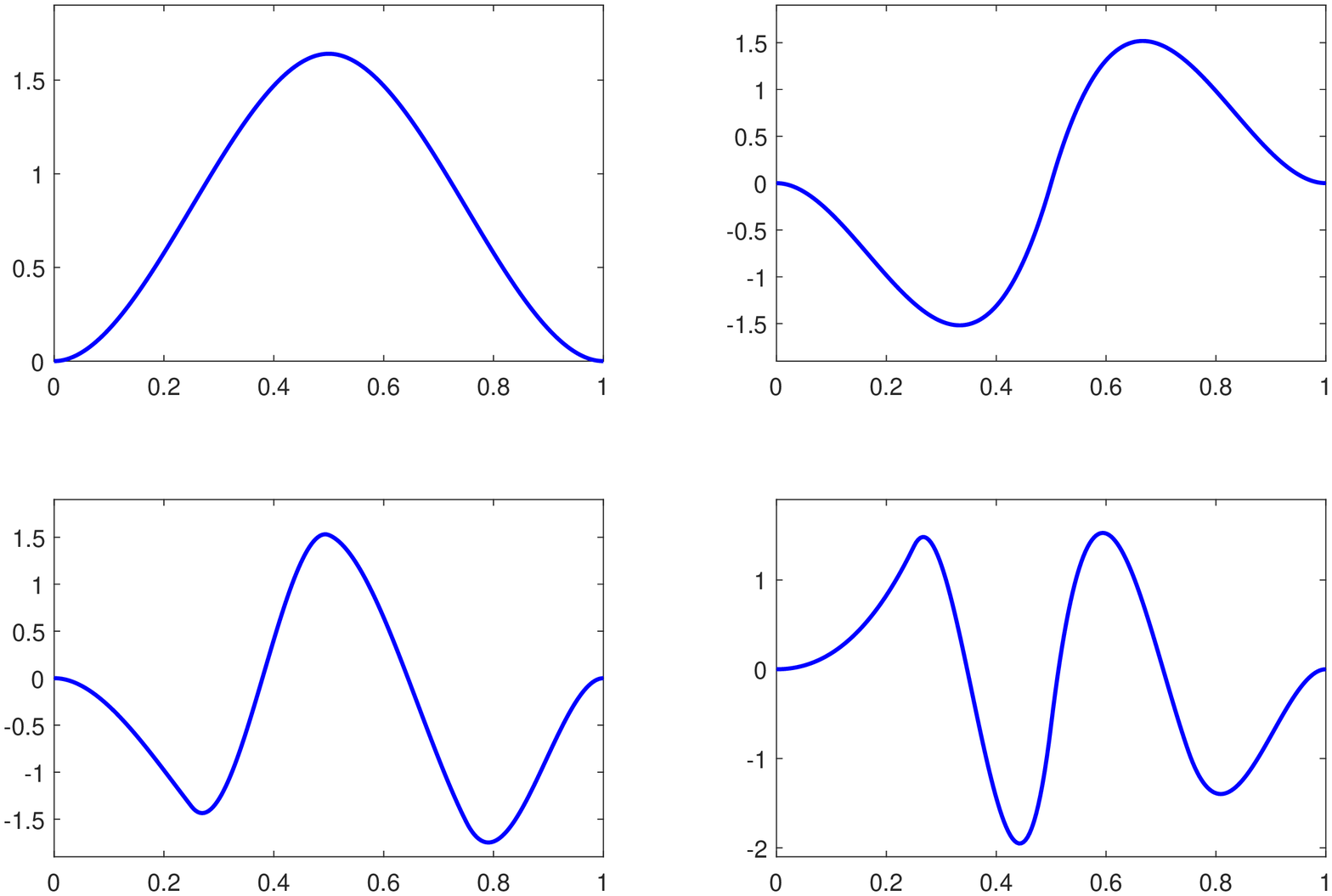} 
\label{Fig2}
\end{figure} 

Lemma~\ref{condition_DGH} states that functions $\phi_1$, $\phi_2$, $\phi_3$, $\phi_4$, $\xi_{1,0}$ and $\xi_{1,1}$ generate spaces $V_j$, which form an orthogonal multiresolution analysis satisfying (\ref{intertwinning_MRA}). However, only functions $\phi_1$, $\phi_2$, $\phi_3$, and $\phi_4$ are orthogonal. Therefore, we perform orthogonalization and normalization to obtain functions generating an orthogonal basis.
Hence, the two remaining generators $\phi_5$ and $\phi_6$ are given by
 \begin{equation}
     g_5 = \xi_{1,0} - \sum_{j=1}^4 \left< \xi_{1,0}, \phi_j \right> \phi_j
     - \sum_{j=1}^4 \left< \xi_{1,0} , \phi_j \left( \cdot + 1 \right) \right> \phi_j \left( \cdot + 1 \right), \quad \phi_5 = \frac{ g_5 }{ \left\|  g_5 \right\| },
 \end{equation}
and 
\begin{equation}
    g_6 = \xi_{1,1}  - \sum_{j=1}^5 \left< \xi_{1,1}, \phi_j \right> \phi_j
    - \sum_{j=1}^5 \left< \xi_{1,1}, \phi_j \left( \cdot + 1 \right) \right> \phi_j \left( \cdot + 1 \right), \quad \phi_6 = \frac{ g_6 }{ \left\| g_6 \right\|}.
\end{equation}
Generators $\phi_5$ and $\phi_6$ supported in $\left[ -1, 1 \right]$
are displayed in Figure~\ref{Fig3}. 

\begin{figure}[ht]
\caption{\emph{ Orthogonal cubic spline generators $\phi_5$ (left) and $\phi_6$ (right) with support $\left[ -1, 1 \right]$.  }}
\includegraphics[width=13.0cm]{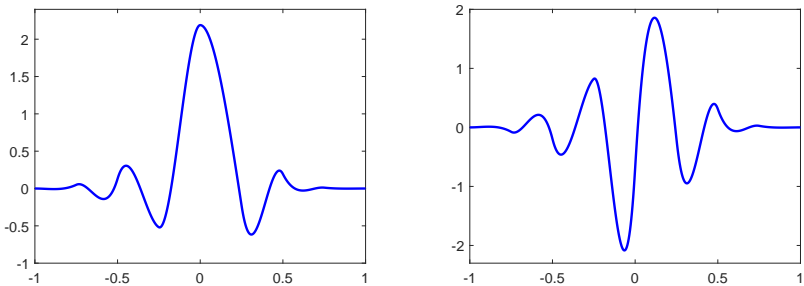} 
\label{Fig3}
\end{figure} 

Using translations and dilations of the generators, we obtain scaling functions 
\begin{eqnarray}  \label{definition_phi_jk}
\phi_{j,6k+l} \left(x\right) &=& 2^{j/2} \phi_l \left( 2^j x - k \right), \ \quad l=1, 2, 3, 4, \quad k \in \mathbb{Z}, \\
\nonumber \phi_{j,6k+l} \left(x\right) &=& 2^{j/2} \phi_l \left( 2^j x - k -1 \right), \ \quad l=5, 6, \quad k \in \mathbb{Z},
\end{eqnarray}
and multiresolution spaces
\begin{equation}
 V_j = \overline{ {\rm span} \left\{ \phi_{j,k}, k \in \mathbb{Z} \right\} }.
\end{equation}

The properties of scaling functions and spaces $V_j$ are summarized in the following  theorem.

\begin{thm} \label{thm_properties_scaling}
Let $\phi_i$, $\phi_{j,k},$ and $V_j$ be defined as above.
\begin{itemize}
    \item[a)] Generators $\phi_i$, $i=1, \ldots, 6$, and scaling functions
    $\phi_{j,k}$, $j, k \in \mathbb{Z}$, are cubic splines belonging to 
    $C^1 \left( \mathbb{R} \right)$.
    \item[b)] The sequence $\left\{ V_j \right\}_{ j \in \mathbb{Z} }$ is an orthogonal multiresolution analysis satisfying~(\ref{intertwinning_MRA}). 
\end{itemize}
\end{thm}

\begin{proof}
\begin{itemize}
    \item[a)] Because all the generators $\phi_i$, $i=1, \ldots, 6,$ are constructed as linear combinations of cubic splines from $C^1 \left( \mathbb{R} \right)$, they are also cubic splines from $C^1 \left( \mathbb{R} \right)$. Definition (\ref{definition_phi_jk}) implies that $\phi_{j,k}$, $j, k \in \mathbb{Z}$ are also cubic splines and belong to $C^1 \left( \mathbb{R} \right)$.
    \item[b)] The fact that the sequence $\left\{ V_j \right\}_{ j \in \mathbb{Z} }$ is an orthogonal multiresolution analysis satisfying (\ref{intertwinning_MRA}) follows from 
    Lemma~\ref{condition_DGH} and the above construction.
\end{itemize}
\end{proof}

As already mentioned, the system of equations (\ref{kvadr_cond}) and (\ref{ortog_cond}) has more than one solution. Moreover, instead of $\phi_1$ and $\phi_2$, we can take their orthogonalized linear combinations and, similarly, instead of $w_i$, we can take their linear combinations and orthogonalize them to obtain different functions $v_i$. Therefore, there are infinitely many choices of generators $\phi_i$, $i=1,\ldots,6.$ We performed extensive numerical experiments to find various solutions. In Appendix~A, we present a variant of generators that leads to a small condition number of the resulting normalized wavelet basis with respect to the $H^1$-seminorm because it also plays a significant role in applications; see \cite{Cohen2003,Cerna2016}.

Furthermore, the part of the construction is the Gram--Schmidt orthogonalization, which can lead to inaccurate results due to round-off errors. Therefore, we computed all the coefficients using quadruple precision and then rounded the results to 16 digits. Computing the scalar products of the resulting functions, we verified numerically that the presented coefficients are sufficiently accurate. For the computation, Matlab and Multiprecision Computing Toolbox Advanpix were used.

\subsection{ Construction of orthogonal cubic spline wavelets }

Wavelet generators corresponding to the multiresolution analysis $\left\{ V_j \right\}_{j \in \mathbb{Z} }$ are constructed as generators of the orthogonal basis of the space $V_1 \ominus V_0$.

The first step is the construction of generators with supports in $I_0$. These generators must be linear combinations of the basis functions of $V_1$ with supports in $I_0$, which are the functions $p_k = \phi_k \left( 2 \cdot \right)$ for $k=1, \ldots ,4$, and  $p_{k+4} = \phi_k \left( 2 \cdot -1  \right)$ for $k=1, \ldots, 6$. The generators must be orthogonal to all functions from $V_0$ with the part or whole support in $I_0$, i.e., to functions $q_i = \phi_i$, $i=1, \ldots, 6,$ and $q_{i+2} = \phi_i \left( \cdot -1 \right)$ for $i=5,6$. Therefore, the coefficients of the linear combinations can be computed as the solution of the homogeneous system of linear equations with the matrix $\mathbf{S}$ with entries $S_{i,j}=\left< p_i, q_j \right>.$ Because functions $p_i$, as well as functions $q_j$, are linearly independent, the null space of the matrix $\mathbf{S}$ is two dimensional. Let vectors $\mathbf{b}^1$ and $\mathbf{b}^2$ form its basis. Then 
\begin{equation}
    u_i = \sum_{j=1}^8 b^i_j \, q_j, \quad i=1,2,
\end{equation}
are functions from $V_1$ with supports in $I_0$. Using the Gram--Schmidt orthogonalization and $L^2$-normalization of $u_1$ and $u_2$, we obtain generators 
\begin{equation}
\psi_1 = \frac{ u_1 }{ \left\| u_1 \right\|}, \quad
\psi_2 = \frac{ u_2 - \left< u_2, \psi_1 \right> \psi_1  }{ \left\| u_2 - \left< u_2, \psi_1 \right> \psi_1 \right\| }.     
\end{equation}
Graphs of functions $\psi_1$ and $\psi_2$ are displayed in Figure~\ref{Fig4}.

\begin{figure}[ht]
\caption{\emph{ Orthogonal wavelet generators $\psi_1$ and $\psi_2$ with support $\left[ 0, 1 \right]$.  }}
\includegraphics[width=13.0cm]{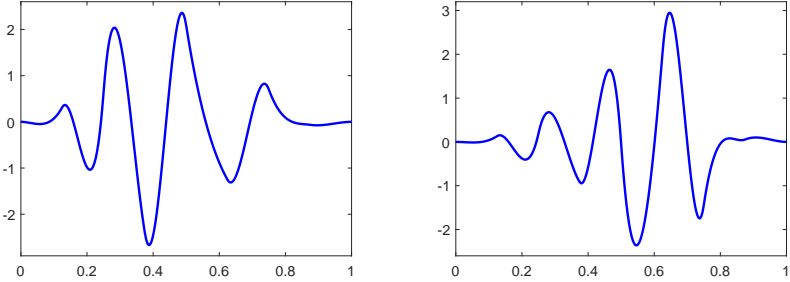} 
\label{Fig4}
\end{figure} 

Now, we construct generators with supports $\left[-1,1\right]$. First, consider functions from $V_1$, which have parts of the support both in $\left( - \infty, 0 \right)$ and $\left( 0, \infty \right)$, i.e., functions $\phi_{1,-1}= 2^{1/2} \phi_5 \left( 2 \cdot \right)$ and $\phi_{1,0} = 2^{1/2} \phi_6 \left( 2 \cdot \right)$. Let us denote
\begin{equation}
    \mathcal{D}_j = \overline{ {\rm span} \left( \left\{ \phi_{j,k}, k \in \mathbb{Z}  \right\} \backslash \left\{ \phi_5 \left( 2^j \cdot \right), \phi_6  \left( 2^j \cdot \right) \right\} \right) }.
\end{equation}
The functions $\phi_5 \left( 2 \cdot \right)$ and $\phi_6 \left( 2 \cdot \right)$ are already orthogonal to the functions from $\mathcal{D}_1$ because $\phi_{2,k}$, $k\in \mathbb{Z}$ are orthogonal. Because $\mathcal{D}_0 \subset \mathcal{D}_1$,  these functions are also orthogonal to the functions from $\mathcal{D}_0$. That is, they are orthogonal to all scaling functions $\phi_{0,k}$, except $\phi_{0,-1} = \phi_5$ and $\phi_{0,0} = \phi_6$. Therefore, functions 
\begin{equation}
    h_{i} = \phi_{i+4} \left( 2 \cdot \right) - 
    \left< \phi_{i+4} \left( 2 \cdot \right), \phi_5 \right> \phi_5 -
    \left< \phi_{i+4} \left( 2 \cdot \right), \phi_6 \right> \phi_6, \quad i=1,2,
\end{equation}
are orthogonal to all scaling functions $\phi_{0,k}$, $k \in \mathbb{Z}.$

Now, orthogonalization and normalization are again performed to obtain wavelet generators
\begin{equation}
\psi_3 = \frac{ h_1 }{ \left\| h_1 \right\|}, \quad
\psi_4 = \frac{  h_2 - \left< h_2, \psi_3 \right> \psi_3  }{ \left\| h_2 - \left< h_2, \psi_3 \right> \psi_3 \right\| }.   
\end{equation}

The remaining two wavelet generators with the support in $\left[ -1, 1 \right]$ are defined as linear combinations of functions from $V_1$ with the whole support in $\left[ -1, 1 \right]$, 
\begin{equation} \label{def_psi_5_6}
    \psi_{i+4} = \sum_{k=-11}^{10} c_{i,k} \phi_{1,k}, \quad i=1,2.
\end{equation}
We need to find coefficients $c_{i,k}$ such that $\psi_5$ and $\psi_6$ are orthogonal to functions from $V_0$, which have part of the support or the whole support in $I_1$, i.e., to functions $\phi_{0,k}$ for $k=-7, \ldots, 6$. Furthermore, these generators must be orthogonal to already constructed wavelets $\psi_i$, $i=1, \ldots, 4,$ and their translations $\psi_i \left( \cdot + 1 \right)$ and $\psi_i \left( \cdot - 1 \right)$ for $i=3, 4$. 

According to the following lemma, we can omit some of these conditions.

\begin{lem}
If functions $\psi_5$ and $\psi_6$ are given by (\ref{def_psi_5_6}) and
$\left\langle \psi_i, \phi_{0,k} \right\rangle$ for $i=5,6$ and $k=-7, \ldots, 6$, then for $i=5,6,$ and $l=3,4$, we have
\begin{equation}
    \left\langle \psi_i, \psi_l \left( \cdot - 1 \right) \right\rangle =0,
    \quad \left\langle \psi_i, \psi_l \left( \cdot + 1 \right) \right\rangle =0.
\end{equation}
\end{lem}

\begin{proof}
According to definitions of generators $\psi_l$, $l=3,4$, we can write \begin{equation} \label{psil_expansion}
\psi_l \left( \cdot - 1 \right) = d_{l,1} \phi_5 \left( 2 \cdot - 2 \right) + d_{l,2} \phi_6 \left( 2 \cdot - 2 \right) + d_{l,3} \phi_5 \left( \cdot - 1 \right) +
d_{l,4} \phi_6 \left( \cdot - 1 \right),
\end{equation}
where $d_{l,i}$ are real coefficients. Assumptions of the lemma imply that
\begin{equation} \label{ortogonality_psi_phi_1}
\left\langle \psi_i,\phi_m \left( \cdot - 1 \right) \right\rangle = 0,
\quad i,m=5,6.
\end{equation}
Due to (\ref{def_psi_5_6}) and orthogonality of functions $\phi_{1,k}$, 
we obtain 
\begin{equation} \label{ortogonality_psi_phi_2}
\left\langle \psi_i, \phi_m \left( 2 \cdot - 2 \right) \right\rangle = \frac{1}{\sqrt{2}} \sum_{k=-11}^{10} c_{i,k} \left\langle \phi_{1,k}, \phi_{1,m+6}  \right\rangle= 0, \quad i,m=5,6.
\end{equation}
Relations (\ref{psil_expansion}), (\ref{ortogonality_psi_phi_1}) and (\ref{ortogonality_psi_phi_2}) imply $\left\langle \psi_i, \psi_l \left( \cdot - 1 \right) \right\rangle =0$ for $i=5,6,$ and $l=3,4$. The proof is similar for functions $\psi_l \left( \cdot + 1 \right)$, $l=3,4$.
\end{proof}

Hence, let $\mathbf{T}$ be a matrix with entries
\begin{eqnarray}
T_{i,j} &=& \left\langle \phi_{0,i-8}, \phi_{1,j-12}  \right\rangle,
i = 1, \ldots, 14, \, j = 1, \ldots, 22, \\
\nonumber T_{i,j} &=& \left\langle \psi_{i-14} \left( \cdot + 1 \right), \phi_{1,j-12}  \right\rangle, i = 15, 16, \, j = 1, \ldots, 22, \\
\nonumber T_{i,j} &=& \left\langle \psi_{i-16}, \phi_{1,j-12}  \right\rangle,
i = 17, \ldots, 20, \, j = 1, \ldots, 22.
\end{eqnarray}
Denote $\mathbf{b}^1$ and $\mathbf{b}^2$ vectors forming a basis of the null space of $\mathbf{T}$ and define
\begin{equation}
z_i = \sum_{k=-11}^{10} b^i_k \, \phi_{1,k}, \quad i=1,2 .  
\end{equation}
Then, we can set
\begin{equation}
\psi_5 = \frac{ z_1 }{ \left\| z_1 \right\|}, \quad
\psi_6 = \frac{ z_2 - \left< z_2, \psi_5 \right> \psi_5  }{ \left\| z_2 - \left< z_2, \psi_5 \right> \psi_5 \right\| }.    
\end{equation}

Graphs of all wavelet generators with supports $\left[ -1, 1 \right]$ are shown in Figure~\ref{Fig5}, and polynomial coefficients of all wavelet generators are given in Appendix~A.

\begin{figure}[ht]
\caption{\emph{ Orthogonal wavelet generators $\psi_3$ (top left), $\psi_4$ (top right), $\psi_5$ (bottom left), and $\psi_6$ (bottom right) with supports $\left[ -1, 1 \right]$.  }}
\includegraphics[width=13.0cm]{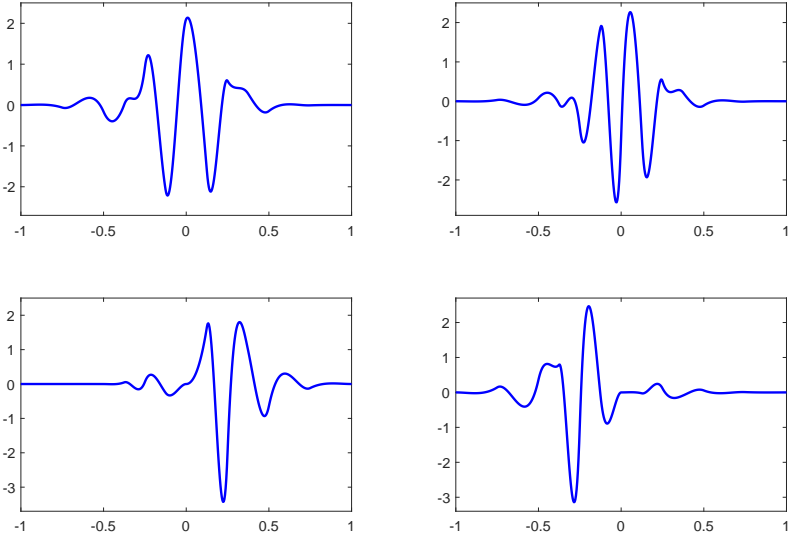} 
\label{Fig5}
\end{figure} 

We obtain wavelets $\psi_{j,k}$ using translations and dilations of the constructed generators, 
\begin{eqnarray}  \label{definition_psi_jk}
\psi_{j,6k+l} \left(x\right) &=& 2^{j/2} \psi_l \left( 2^j x - k \right), \ \quad l=1, 2, \quad j, k \in \mathbb{Z}, \\
\nonumber \psi_{j,6k+l} \left(x\right) &=& 2^{j/2} \psi_l \left( 2^j x - k - 1 \right), \ \quad l=3, 4, 5, 6, \quad j, k \in \mathbb{Z}.
\end{eqnarray}
These wavelets form wavelet spaces 
\begin{equation}
W_j = \overline{ {\rm span} \left\{ \psi_{j,k}, k \in \mathbb{Z} \right\} }, \quad j \in \mathbb{Z}.
\end{equation}

The following theorem summarizes the properties of the constructed wavelets and wavelet spaces and shows that
\begin{equation} \label{def_Psi}
    \Psi = \left\{ \phi_{0,k}, k \in \mathbb{Z} \right\} \cup \left\{ \psi_{j,k}, j \geq 0, k \in \mathbb{Z} \right\}
\end{equation}
is a wavelet basis of the space $L^2 \left( \mathbb{R} \right).$ The symbol
$\delta_{i,j}$ denotes the Kronecker delta, i.e., $\delta_{ii}=1$ and $\delta_{ij}=0$ for $i \neq j$. 

\begin{thm} \label{properties_psi_on_R}
Under the above setting, the following statements hold:
\begin{itemize}
    \item[a)] Generators $\psi_i$, $i=1, \ldots, 6$, and wavelet
    $\psi_{j,k}$, $j, k \in \mathbb{Z}$, are cubic splines belonging to 
    $C^1 \left( \mathbb{R} \right)$.
    \item[b)] For $i=1, \ldots, 6,$ scaling generators $\phi_i$ and wavelet generators $\psi_i$ satisfy
    \begin{equation}
        \left\langle \phi_i, \phi_j \right\rangle = \delta_{i,j}, \quad
        \left\langle  \psi_i, \psi_j \right\rangle = \delta_{i,j}, \quad  \left\langle \phi_i, \psi_j \right\rangle = 0.
    \end{equation}
    \item[c)] For $i,j,k,l \in \mathbb{Z},$ scaling functions $\phi_{i,k}$ and wavelets 
    $\psi_{j,l}$ satisfy orthogonality conditions
    \begin{equation} \label{ortog_phi_ik_psi_jl_real}
        \left\langle \phi_{i,k}, \phi_{i,l} \right\rangle = \delta_{k,l}, \quad
        \left\langle  \psi_{i,k}, \psi_{j,l} \right\rangle = \delta_{i,j} \delta_{k,l}, 
    \end{equation}
    and if $i \leq j$, then
    \begin{equation} \label{ortog_phi_ik_psi_jl}
        \left\langle \phi_{i,k}, \psi_{j,l} \right\rangle = 0. 
    \end{equation}
    \item[d)] Wavelets $\psi_{j,k}$, $j, k \in \mathbb{Z}$, have four vanishing moments.
    \item[e)] The set $\Psi$ is an orthogonal basis of the space $L^2 \left( \mathbb{R} \right).$
    \item[f)] Basis functions from $\Psi$ are local in a sense that
    ${\rm diam \, supp } \, \phi_{0,k} \leq 2$ and ${\rm diam \, supp } \, \psi_{j,k} \leq 2^{1-j}$.
\end{itemize}
\end{thm}

\begin{proof}
\begin{itemize}
    \item[a)] Because $\psi_{j,k}$, $j,k \in \mathbb{Z}$ are linear combinations of scaling functions, the assertion a) is valid.
    \item[b)] This statement follows directly from the construction.
    \item[c)] The orthogonality of scaling functions $\phi_{i,k}$ follows from b) and (\ref{definition_phi_jk}). Similarly, orthogonality of wavelets $\psi_{i,k}$ and $\psi_{i,l}$ follows from b) and (\ref{definition_psi_jk}). Due to $\left\langle \phi_i, \psi_j \right\rangle = 0$ in b), $V_j \perp W_j$. Since $V_j \perp W_j$ and $V_i \subset V_j$ and $W_i \subset V_j$ for $i \leq j$, we have $V_i \perp W_j$ and $W_i \perp W_j$ for $i \leq j$, which implies (\ref{ortog_phi_ik_psi_jl_real}) and  (\ref{ortog_phi_ik_psi_jl}). 
    \item[d)] Due to c), wavelets $\psi_{j,k}$, $j,k \in \mathbb{Z}$ are orthogonal to all functions from $V_j$, which is a cubic spline space containing all cubic polynomials. 
    The orthogonality of $\psi_{j,k}$ to all cubic polynomials means that wavelets have four vanishing moments.
    \item[e)] The orthogonality of $\Psi$ follows from d). By Theorem~\ref{thm_properties_scaling}~b), $\Psi$ generates $L^2 \left( \mathbb{R} \right)$. 
    \item[f)] The length of the support of generators $\phi_i$ and $\psi_i$, $i=1, \ldots, 6,$ is at most $2$. Hence, due to 
    (\ref{definition_phi_jk}) and (\ref{definition_psi_jk}), the assertion is valid.
\end{itemize}
\end{proof}

Due to Theorem~\ref{properties_psi_on_R}, all the requirements $A1)-A5)$ are met and thus $\Psi$ is indeed an orthogonal wavelet basis of  $L^2 \left( \mathbb{R} \right).$

\subsection{ Orthogonal wavelet basis on the interval}

In this section, we adapt the constructed wavelet basis from the real line to a bounded interval and homogeneous Dirichlet boundary conditions, and study the properties of the resulting basis. 

It is sufficient to focus on the interval $\left[ 0, 1 \right]$ because a wavelet basis on another bounded interval can be obtained by a linear transformation. The adaptation consists of preserving those basis functions whose whole supports are contained in $\left[ 0, 1 \right]$ and constructing new boundary functions.

First, we investigate the situation at the point $0$. The functions $\phi_5$ and $\phi_6$ are the only two scaling functions on the level zero with nonzero value at $0$. Because we want to obtain a new generator that is orthogonal to other functions on $\left[ 0, 1 \right]$ and is zero at $0$, we define this boundary scaling generator as
\begin{equation}
    \phi_L = \left( \phi_6 \left( 0 \right) \phi_5 -  \phi_5 \left( 0 \right) \phi_6 \right)  \chi_{\left[ 0 , 1 \right]}.
\end{equation}
Using a similar approach at the point $1$ leads to a boundary scaling function
\begin{equation}
    \phi_R = \left( \phi_6 \left( 0 \right) \phi_5 \left( \cdot -1 \right)  -  \phi_5 \left( 0 \right) \phi_6 \left( \cdot -1 \right) \right) \chi_{\left[ 0 , 1 \right]}.
\end{equation}
The graphs of boundary scaling generators are displayed in Figure~\ref{Fig6}.

\begin{figure}[ht]
\caption{\emph{ Boundary orthogonal scaling generators $\psi_L$ (left) and $\psi_R$ (right) with the support $\left[ 0, 1 \right]$.  }}
\includegraphics[width=13.0cm]{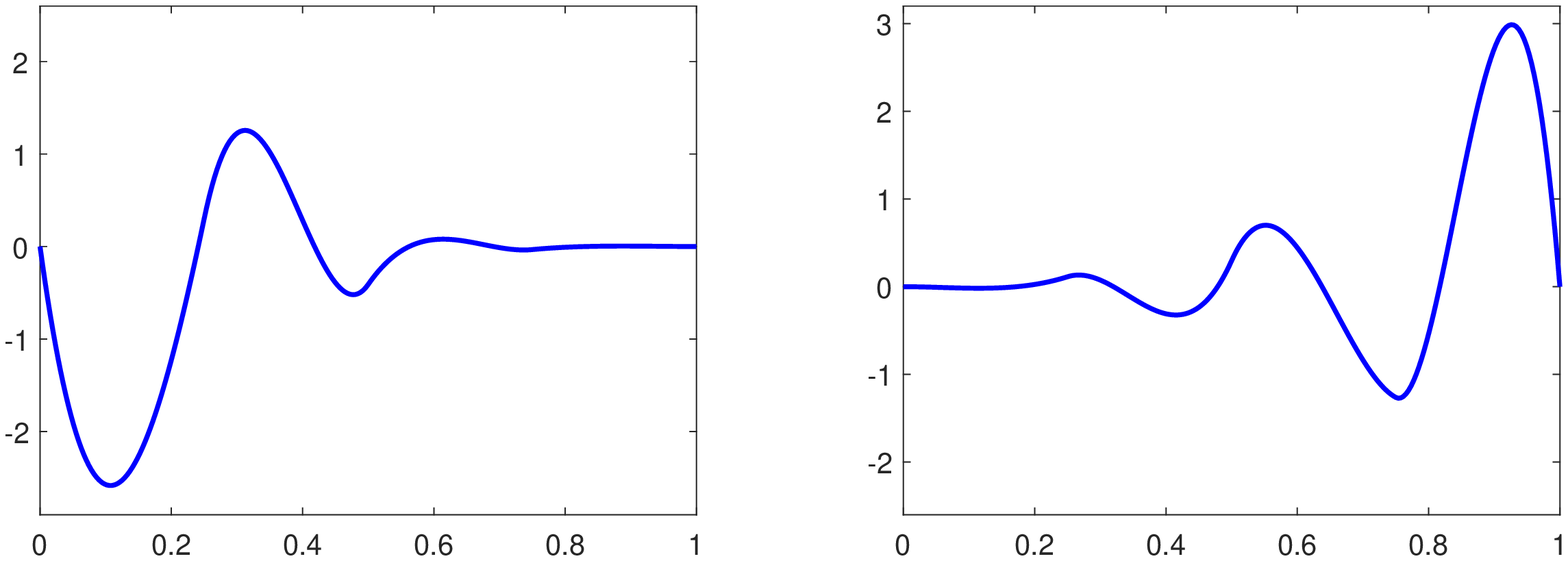} 
\label{Fig6}
\end{figure} 

When it comes to boundary wavelet generators at the point $0$, the situation is similar in the sense that new boundary wavelets are linear combinations of wavelets that are nonzero at $0$. In addition, the new wavelet generators must be orthogonal to $\phi_L$. Therefore, we define functions $g_{L_1}$ and $g_{L_2}$ as
\begin{equation}
    g_{L_i} = \sum_{j=1}^4 b_j^i \, \psi_{j+2}, \quad i=1,2.
\end{equation}
Because we require $g_{L_i} \left( 0 \right) = 0$ and $\left\langle \phi_L, g_{L_i} \right\rangle = 0$, the coefficients $b_j^i$ are elements of vectors $\mathbf{b}^i$ forming a basis of the null space of the matrix $\mathbf{G}$ given by
\begin{equation}
    G_{1,j} = \psi_{j+2} \left( 0 \right), \quad 
    G_{2,j} = \left< \phi_L, \psi_{j+2} \right>, \quad j=1, \ldots, 4.
\end{equation}
Generators $\psi_{L_1}$ and $\psi_{L_2}$ are obtained by orthogonalization and normalization of $g_{L_1}$ and $g_{L_2}$, similarly to the generators above.

The construction of right boundary wavelet generators is analogous. The functions $h_{R_1}$ and $h_{R_2}$ are defined as
\begin{equation}
    h_{R_i} = \sum_{j=1}^4 c_j^i \, \psi_{j} \left( \cdot - 1 \right), \quad i=1,2,
\end{equation}
where vectors $\mathbf{c}^i$ form the null space of matrix $\mathbf{H}$ with entries
\begin{equation}
   H_{1,j} = \psi_{j+2} \left( 0 \right), \quad  H_{2,j} = \left< \phi_R, \psi_{j+2} \left( \cdot - 1 \right) \right>, \quad j=1, \ldots, 4.
\end{equation}
Generators $\psi_{R_1}$ and $\psi_{R_2}$ are obtained by orthogonalization and normalization of $h_{R_1}$ and $h_{R_2}$. The boundary wavelet generators are displayed in Figure~\ref{Fig7}.

\begin{figure}[ht]
\caption{\emph{ Boundary orthogonal wavelet generators $\psi_{L_1}$ (top left), $\psi_{L_2}$ (top right), $\psi_{R_1}$ (bottom left), and $\psi_{R_2}$ (bottom right) with the support $\left[ 0, 1 \right]$.  }}
\includegraphics[width=13.0cm]{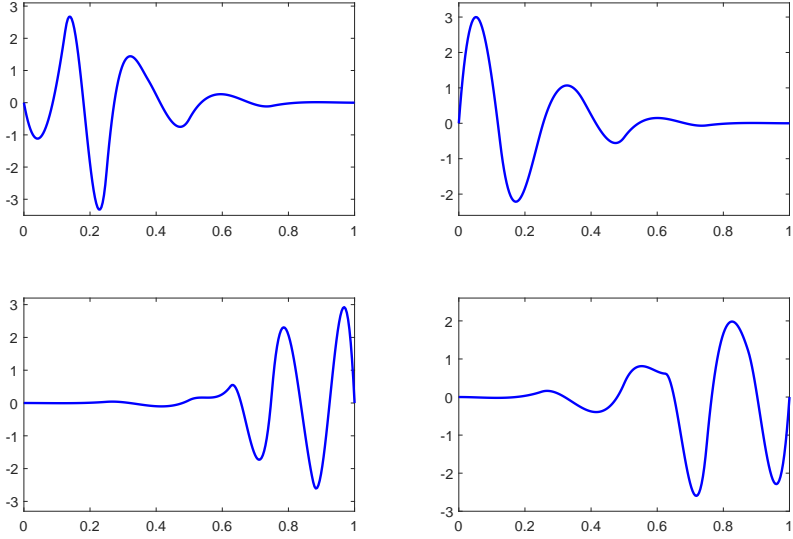} 
\label{Fig7}
\end{figure} 

Let us define scaling functions on the level $0$ as $\phi_{0,1} = \phi_L$, $\phi_{0,6}=\phi_R$, and $\phi_{0,k} = \phi_{k-1}$ for $k=2, \ldots, 5$.
On each level $j \geq 0$, we define boundary wavelets
\begin{equation}
\psi_{j,i} = \psi_{L_i} \left( 2^j \cdot \right),  \quad
\psi_{j,6 \cdot 2^j +i -2} = \psi_{R_i} \left( 2^j \left( 1 - \cdot \right) \right), \quad i=1,2.
\end{equation}
and inner wavelets
\begin{equation}
\psi_{j,6k+l+2} \left(x\right) = 2^{j/2} \psi_l \left( 2^j x - k \right), \quad l=1, \ldots, 6, \quad 3 \leq 6k+l \leq 6 \cdot 2^j -2.
\end{equation}

We define an index set $\mathcal{J}_j = \left\{ k \in \mathbb{Z}, 1 \leq k \leq 6 \cdot 2^j \right\}$ and assign the constructed functions to sets
\begin{equation}
    \Phi_0 = \left\{ \phi_{0,k}, k \in \mathcal{J}_0 \right\}, \quad
    \Psi_j = \left\{ \psi_{j,k}, k \in \mathcal{J}_j \right\},
\end{equation}
and 
\begin{equation} \label{definition_Psi_I}
    \Psi^I = \Phi_0 \cup \bigcup_{j=0}^{\infty} \Psi_j, \quad
    \Psi^k = \Phi_0 \cup \bigcup_{j=0}^{k} \Psi_j.
\end{equation}

Our aim is to show that $\Psi^I$ is an orthogonal wavelet basis of the space $L^2 \left( 0, 1 \right).$ The properties $A1)-A5)$ are studied in the following theorem.

\begin{thm} \label{properties_psi_on_I}
Under the above setting, the following statements hold:
\begin{itemize}
    \item[a)] All basis functions from $\Psi^I$ are cubic splines belonging to 
    $C^1 \left( 0, 1 \right)$.
    \item[b)] Basis functions from $\Psi^I$ satisfy orthogonality conditions
    \begin{equation}
        \left\langle \phi_{0,k}, \phi_{0,l} \right\rangle = \delta_{k,l}, \quad
        \left\langle  \psi_{i,k}, \psi_{j,l} \right\rangle = \delta_{i,j} \delta_{k,l}, \quad
        \left\langle \phi_{0,k}, \psi_{j,l} \right\rangle = 0. 
    \end{equation}
    \item[c)] Inner wavelets $\psi_{j,k}$, $k = 3, \ldots, 6 \cdot 2^j -2$ have four vanishing moments.
    \item[d)] The set $\Psi^I$ is an orthogonal basis of $L^2 \left( 0, 1 \right).$
    \item[f)] Basis functions from $\Psi^I$ are local in the sense that
    ${\rm diam \, supp } \, \phi_{0,k} \leq 2$ and ${\rm diam \, supp } \, \psi_{j,k} \leq 2^{1-j}$.
\end{itemize}
\end{thm}

\begin{proof}
The proof is analogous to the proof of Theorem~\ref{properties_psi_on_R}.
\end{proof}

Due to Theorem~\ref{properties_psi_on_I}, we can conclude that $\Psi^I$ is an orthogonal wavelet basis of $L^2 \left( 0, 1 \right).$ It remains to show that the assumption $A6)$ is also satisfied. To this end, we employ the following theorem. For the proof, see  \cite{Cohen2003, Dahmen1996, Dijkema2010}. 

\begin{thm} \label{norm_equivalences}
Let $j_0 \in \mathbb{N}$ and for $j\geq j_0$ let $V_j$ and $\tilde{V}_j$ be subspaces of the space $H \subset L^2 \left( I \right)$ such that $V_j \subset V_{j+1}$, $\tilde{V}_j \subset \tilde{V}_{j+1}$, and ${\rm dim} \ V_j = {\rm dim} \ \tilde{V}_j < \infty$. Let $\Phi_j$ be bases of $V_j$, $\tilde{\Phi}_j$ be bases of $\tilde{V}_j$, and $\Psi_j$ be bases of $\tilde{V}_j^{\perp} \cap V_{j+1}$, where $\tilde{V}_j^{\perp}$ denotes the $L^2$-orthogonal complement of $\tilde{V}_j$ in $H$. Moreover, let the Riesz bounds with respect to the $L^2$-norm of $\Phi_j$, $\tilde{\Phi}_j$, and $\Psi_j$ be uniformly bounded. Let $\Psi$ be composed of $\Phi_{j_0}$ and $\Psi_j$, $j \geq j_0$,  as in (\ref{definition_Psi}).
Furthermore, we assume that
\begin{equation} \label{def_Gj}
\mathbf{\Gamma}_j  = \left\langle \Phi_j, \tilde{\Phi}_j \right\rangle
\end{equation}
is invertible and that the spectral norm of $\mathbf{\Gamma}_j^{-1}$ is bounded independently on~$j$. In addition, for some positive constants $C$, $\gamma$ and $d$, such that $\gamma<d$, let
\begin{equation} \label{jackson}
\inf_{v_j \in V_j} \left\| v - v_j \right\| \leq C 2^{-j t } \left\|v \right\|_{ t }, \ v \in H^t \left( I \right) \cap H, \ 0 \leq t \leq d,
\end{equation}
and 
\begin{equation} \label{bernstein}
\left\|v_j \right\|_{ s } \leq C 2^{js}  \left\| v_j \right\|, \ v_j \in V_j,
\ 0\leq s < \gamma,
\end{equation}
and, similarly, let (\ref{jackson}) and (\ref{bernstein}) hold for $\tilde{\gamma}$ and $\tilde{d}$ on the dual side. Then 
\begin{equation}
\left\{ \psi_{\lambda} / \left\| \psi_{\lambda} \right\|_{ s }, \psi_{\lambda} \in \Psi \right\}
\end{equation}
is a Riesz sequence in $H^s \left( I \right)$ for $s \in \left(- \tilde{\gamma}, \gamma \right)$.
\end{thm}

Estimate (\ref{jackson}) characterizes the approximation properties of spaces $V_j$ and is often referred to as the {\it Jackson estimate}. Estimate (\ref{bernstein}) relates to the smoothness properties of $V_j$, and is called the {\it Bernstein estimate}.

\begin{thm} \label{basis_Hs}
For $s \in \left( -1.5, 1.5 \right)$, the set $\left\{ \psi_{\lambda} / \left\| \psi_{\lambda} \right\|_{ s }, \psi_{\lambda} \in \Psi^I \right\}$ is a Riesz sequence in $H^s$, and the set $\left\{ \psi_{\lambda} / \left\| \psi_{\lambda} \right\|_{ 1 }, \psi_{\lambda} \in \Psi^I \right\}$ is a Riesz basis of $H_0^1 \left(0, 1 \right).$
\end{thm}

\begin{proof}
Due to the orthogonality, matrix $\mathbf{\Gamma}_j$ defined by (\ref{def_Gj}) is an identity matrix.
The estimates (\ref{jackson}) and (\ref{bernstein}) are satisfied for $d = \tilde{d}=4$ and $\gamma=\tilde{\gamma}=1.5$. These parameters depend
on the polynomial exactness and smoothness of the primal and dual spaces; see \cite{Cohen2003,Dahmen1996}. Due to these facts and Theorem~\ref{norm_equivalences}, the proof is complete.
\end{proof}

As already mentioned, we performed extensive numerical experiments to construct an $L^2$-orthogonal cubic spline wavelet basis on the interval adapted to homogeneous boundary conditions, which has also small condition number with respect to the $H^1$-seminorm. This basis is presented in~\ref{Appendix_1}. In Figure~\ref{Fig8}, the condition numbers of the set $\left\{ \psi_{\lambda} / \left| \psi_{\lambda} \right|_{ 1 }, \psi_{\lambda} \in \Psi^k \right\}$ with respect to $H^1$-seminorm are presented. These condition numbers are square roots of the condition numbers of stiffness matrices for these sets and are close to the condition numbers with respect to the $H^1$-norm.

\begin{figure}
    \centering
    \includegraphics[width=6.0cm]{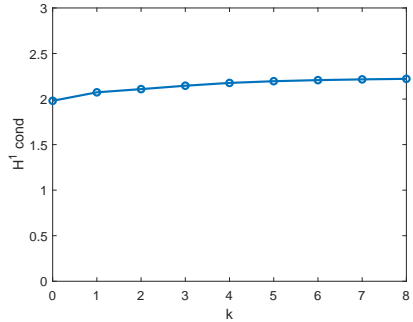}
    \caption{Condition numbers with respect to the $H^1$-seminorm for the sets $\Psi^k$ normalized in the $H^1$-seminorm. }
    \label{Fig8}
\end{figure}

\subsection{ Orthogonal wavelets on sparse tensor product spaces }
\label{sparse_tensor_product_wavelets}

Wavelet bases on the unit cube $\Omega$ can be created using the tensor product of wavelet bases on the interval $\left( 0, 1 \right)$. This can be achieved by several approaches, such as an isotropic approach \cite{Cerna2021a}, an anisotropic approach \cite{Cerna2020,Cerna2021a}, or the sparse tensor product \cite{Hilber2013}. When solving higher-dimensional problems, the isotropic and anisotropic approaches lead to an excessive number of functions $N \sim \mathcal{O} \left( 2^{dk} \right)$ at higher levels $k$, and $L^2$ error of approximation is of order $\mathcal{O} \left( N^{-p/d} \right)$, where $p$ is order of spline wavelet. This phenomenon, known as the ``curse of dimensionality," can be circumvented by using sparse grids  \cite{Griebel1995,Griebel1999,Hilber2013,Petersdorff2004,Reich2008}. In this case, the number of basis functions is $\mathcal{O} \left( 2^{k} k^{d-1} \right)$ and the $L^2$--error of approximation is the same as in one dimension, up to log terms. Therefore, we use the sparse tensor product approach here.

Let $\Psi^I$ be a wavelet basis on the interval $I=\left[ 0 , 1 \right]$ given by (\ref{definition_Psi_I}). To simplify notations redefine  $\mathcal{J}_{0} = \left\{ -5, \ldots, 6 \right\}$ and denote
$\psi_{0,k-6}=\phi_{0,k}$ for $k = 1, \ldots, 6.$
Then the set of all possible indices for basis functions is 
\begin{equation}
\mathcal{J} =\left\{ \left(j,k\right), \ j \geq 0, \ k \in \mathcal{J}_j   \right\},
\end{equation}
and the wavelet basis $\Psi^I$ can also be expressed as
\begin{equation}
\Psi^I=\left\{ \psi_{j,k}, \ j \geq 0, \ k \in \mathcal{J}_j \right\}= \left\{ \psi_{\lambda}, \ \lambda \in \mathcal{J} \right\}.
\end{equation}

Now, consider dimension $d \geq 1$. The index set is defined as
\begin{equation}
    \mathcal{J}= \left\{ \lambda= \left( \lambda_1, \ldots, \lambda_d \right): \lambda_i =
    \left( j_i, k_i \right), j_i \geq 0, k_i \in \mathcal{J}_{j_i} \right\}.
\end{equation}
On the unit cube $\Omega$, we consider a set $\Psi$ containing functions 
\begin{equation}
\psi_{\lambda } = \bigotimes\limits_{i=1}^d \psi_{\lambda_i}, \quad \lambda= \left( \lambda_1, \ldots, \lambda_d \right) \in \mathcal{J},
\end{equation}
where $\otimes$ denotes the tensor product, i.e., $\left( u \otimes v \right) \left(x_1, x_2\right)=u \left(x_1\right) v\left(x_2\right)$. 

The set of all such functions is $\Psi = \left\{ \psi_{\lambda}, \lambda \in \mathcal{J} \right\}$. We show below that generalizations of A1) -- A6) are valid and, therefore, this set is also referred to as {\it a wavelet basis}. Recall that the level of one-dimensional function $\psi_{\lambda}$ is denoted $\left| \lambda \right|$. For $k\geq 0$, the so-called sparse tensor product approach uses the set of functions
\begin{equation} \label{Psi_k_sparse}
\Psi^k=\left\{ \psi_{ \lambda }:
\lambda = \left( \lambda_1, \ldots, \lambda_d \right), \ 
\sum_{i=1}^d \left| \lambda_i \right| \leq k \right\},
\end{equation}
whereas the so-called anisotropic tensor product approach employs the set
\begin{equation} 
\Psi^k_{ani}=\left\{ \psi_{ \lambda }:
\lambda = \left( \lambda_1, \ldots, \lambda_d \right), \ 
\left| \lambda_i \right| \leq k \right\}.
\end{equation}

These two kinds of sets significantly differ in a number of functions, as shown in the following lemma. Symbol $\#$ denotes the cardinality.
\begin{lem} \label{lemma_number_sparse}
For $d \geq 1$ and $k \geq 1$, 
\begin{equation}
\# \Psi^k_{ani} = 6^d 2^{dk} = \mathcal{O} \left( 2^{dk} \right), \quad
\# \Psi^k = \mathcal{O} \left( 2^k k^{d-1} \right).
\end{equation}
More precisely, 
\begin{eqnarray}
\# \Psi^k &=& 6^2 \left( k + 2 \right) 2^{k+1}, \ \text{if} \ d=2, \\
\nonumber \# \Psi^k &=& 6^3 \left( k^2+7k+8 \right) 2^{k}, \ \text{if} \ d=3, \\
\nonumber \# \Psi^k &=& 6^4 \left( \frac{k^3}{6} + \frac{5k^2}{2} + \frac{28k}{3} + 8 \right) 2^{k+1}, \ \text{if} \ d=4, \\
\nonumber \# \Psi^k &=& 6^5 \left( \frac{k^4}{24} + \frac{(13 k^3}{12} + \frac{ 203 k^2}{24} + \frac{269 k}{12} + 16 \right) 2^{k+1}, \ \text{if} \ d=5. 
\end{eqnarray}
\end{lem}

\begin{proof}
Clearly, $\# \Psi^k_{ani} = 6 \cdot 2^k$ for $d=1$ and $k \geq 1$, and thus $\# \Psi^k_{ani} = \left( 6 \cdot 2^k \right)^d = 6^d 2^{dk}$ for $d \geq 1$ and $k \geq 1$.

For $d=2$, 
\begin{equation}
\# \Psi^k = \sum\limits_{i=0}^k \sum\limits_{j=0}^{k-i}  \# \mathcal{J}_i \# \mathcal{J}_j =  \# \mathcal{J}_0 \# \mathcal{J}_0 + 2 \# \mathcal{J}_0 \sum_{j=1}^k \# \mathcal{J}_j + \sum\limits_{i=1}^{k-1} \sum\limits_{j=1}^{k-i}  \# \mathcal{J}_i \# \mathcal{J}_j.    
\end{equation}
Using $\# \mathcal{J}_0 = 12$, $\# \mathcal{J}_j = 6 \cdot 2^j$ for $j \geq 1$, and the formula for a sum of geometric series gives the relation for $\# \Psi^k$ for $d=2$.

A similar approach used for $d > 2$ gives the remaining assertions of the lemma.
\end{proof}

The resulting multidimensional functions preserve many properties of original univariate functions.

\begin{thm}
Under the above setting, the following statements hold:
\begin{itemize}
\item[a)] Functions $\psi_{\lambda}$ are bicubic splines belonging to $C^1 \left( \Omega \right)$.
\item[b)] The set $\Psi$ is an orthogonal basis of $L^2 \left( \Omega \right)$.
\item[c)] Functions $\psi_{\lambda}$ satisfy $\left\langle p, \psi_{\lambda} \right\rangle = 0$, where
$\lambda= \left( \lambda_1, \ldots, \lambda_d \right)$ such that
$\lambda_i = \left( j_i, k_i \right)$, $3 \leq k_i \leq 6 \cdot 2^{j_i} - 2$,
and $p$ is a polynomial of degree less than $4$, which means that inner wavelets have four vanishing moments.
\item[d)] Basis functions are local in the generalized sense that 
\begin{equation}
  {\rm diam} \, {\rm supp} \, \psi_{ \lambda } \leq C \sqrt{d} \, 2^{- \left[ \lambda  \right] }, \quad    \left[ \lambda \right]= \min_{i=1, \ldots, d} \left| \lambda_i \right|.
\end{equation}
\end{itemize}
\end{thm}

\begin{proof}
The properties a), b), and c) follow directly from a definition of $\psi_{\lambda}$ and properties of the corresponding univariate functions presented in Theorem~\ref{properties_psi_on_I}.

Property d) follows from
\begin{equation} \label{delka_nosice_nD}
 {\rm diam} \, {\rm supp} \, \psi_{ \lambda } \leq \sqrt{ \sum_{i=1}^d C_i^2 2^{-2 \left| \lambda_i \right| } }\leq C \sqrt{d} 2^{- \left[ \lambda  \right] }, \quad C=\max \limits_{i=1, \ldots, d} C_i.
 \end{equation}
\end{proof}

Hence, conditions $A1)-A5)$ are fulfilled in a generalized sense. The next theorem studies condition $A6)$.

\begin{thm}
The set $\left\{ \psi_{\lambda} / \left\| \psi_{\lambda} \right\|_{ 1 }, \psi_{\lambda} \in \Psi \right\}$ is a Riesz basis of
$H_0^1 \left( \Omega \right)$. 
\end{thm}

\begin{proof}
The space $H_0^1 \left( \Omega \right)$ can be written as
\begin{equation} \label{form_space_H_0^1}
H_0^1 \left( \Omega \right) =  \bigcap\limits_{i=1}^d S^i \left( \Omega \right),
\end{equation}
where
\begin{equation}
S^i \left( \Omega \right) =  L^2  \left( I_1 \right) \otimes \ldots \otimes L^2 \left( I_{i-1} \right) \otimes H_0^1 \left( I_i \right) \otimes L^2 \left( I_{i+1} \right) \otimes \ldots L^2  \left( I_d \right),
\end{equation}
see \cite{Griebel1995}. Thus, Theorem~\ref{basis_Hs} and (\ref{form_space_H_0^1}) yield the required result.
\end{proof}

Because $\Psi$ is an $L^2$--orthogonal basis of $L^2 \left( \Omega \right)$, any function $u \in L^2 \left( \Omega \right)$ has a series representation
\begin{equation}
    u = \sum_{ \lambda \in \mathcal{J} } c_{ \lambda } \psi_{ \lambda }.
\end{equation}
Truncating this expansion, we obtain a projection operator $P_k$,
\begin{equation}
  P_k \, u = \sum_{ \psi_{\lambda} \in \Psi^k }  c_{ \lambda } \psi_{ \lambda }, \quad u \in  L^2 \left( \Omega \right).
\end{equation}

The approximation properties of wavelet sparse grids have been studied in \cite{Griebel1999,Hilber2013,Petersdorff2004,Reich2008}. For the proof of the next result, refer therein. 

\begin{thm} \label{thm_approx_sparse}
Let $\Psi^k$ given by (\ref{Psi_k_sparse}) be set of function constructed by the tensor product of univariate spline-wavelet basis of order $p$ satisfying $A1)-A5)$. Then, for $u \in H^p \left( \Omega \right)$ there exists a constant $C$ independent of level $k$ such that
\begin{equation} \label{sparse_approx}
    \left\| u - P_k u \right\| \leq C 2^{-pk} k^{ \frac{d-1}{2} }
    \left\| u \right\|_p.
\end{equation}
\end{thm}

Hence, the approximation order is $\mathcal{O} \left( N^{-p} \right)$, if dimension $d=1$. Using Lemma~\ref{lemma_number_sparse} and (\ref{sparse_approx}), we may conclude that, as mentioned above, the approximation rate for $d$-dimensional functions is indeed, up to the log term, the same as in one dimension, and the curse of dimensionality is overcome. 

\medskip
\section{ Wavelet Method }
\label{Section_method}

This section examines a wavelet-based method for solving equation~(\ref{eq_bounded_domain_2}) equipped with conditions (\ref{init_condition}) and (\ref{boundary_conditions}). The method is based on the Galerkin method using sparse tensor product spaces of orthogonal spline wavelets, such as cubic spline wavelets from Section~\ref{sparse_tensor_product_wavelets}, and the Crank--Nicolson scheme with Rannacher time stepping.

For $u,v \in H_0^1 \left( \Omega \right)$, the bilinear form $a: H_0^1 \left( \Omega \right) \times H_0^1 \left( \Omega \right) \rightarrow \mathbb{R}$ corresponding to the differential operator $\mathcal{D}$ given by (\ref{definition_operator_D}) is defined as 
\begin{equation} \label{definition_bilinear_form_a}
a\left( u, v \right) = \sum_{i=1}^d \sum_{j=1}^d P_{i,j} \left\langle \frac{ \partial u }{\partial x_i}, \frac{ \partial v }{\partial x_j} \right\rangle +  r \left\langle u, v \right\rangle.
\end{equation}

\begin{lem} \label{Lemma_Garding} 
The bilinear form $a$ defined by (\ref{definition_bilinear_form_a}) is continuous and coercive; that is, there exists a constant $\beta$ such that
\begin{equation} \label{def_continuity}
  \left|  a \left( u, v \right) \right| \leq \beta \left\| u \right\|_{ 1 } \left\| v \right\|_{ 1 } \quad \forall \, u, v \in H_0^1 \left( \Omega \right),
\end{equation}
and a constant $\alpha > 0$ such that
\begin{equation} \label{coercivity}
    a \left( u, u \right) \geq \alpha \left\| u \right\|^2_{1} \quad \forall \, u \in H_0^1 \left( \Omega \right).
\end{equation}
Specifically, the form $a$ satisfies the G\aa rding inequality; that is, there exist constants $C_1, C_2 >0$ such that, for any $v \in H_0^1 \left( \Omega \right)$,
\begin{equation} \label{garding}
a \left( v, v \right) \geq C_1 \left\| v \right\|_{1}^2 - C_2  
\left\| v \right\|^2.
\end{equation}
\end{lem}

\begin{proof}
The continuity of the bilinear form $a$ is a consequence of definition (\ref{definition_bilinear_form_a}) and the Cauchy--Schwarz inequality. The coefficient matrix for the second order terms of differential operator $\mathcal{D}$ is $\mathbf{D} \mathbf{Q} \mathbf{D} / 2$, where $\mathbf{D}$ is a diagonal matrix with entries $D_{i,i}= 1/ d_i$, $d_i = x_i^{max}-x_i^{min}$, and $\mathbf{Q}$ is a matrix defined in Section~\ref{Section_model}. Because $\mathbf{Q}$ is a variational matrix, it is positive definite, and thus for any $0 \neq \mathbf{x} \in \mathbb{R}^d$,
\begin{equation}
\mathbf{x}^T \mathbf{D} \mathbf{Q} \mathbf{D} \mathbf{x} =
\left( \mathbf{D} \mathbf{x} \right)^T \mathbf{Q} \left( \mathbf{D} \mathbf{x} \right) > 0.
\end{equation}
This implies that $\mathbf{D} \mathbf{Q} \mathbf{D} / 2$ is also positive definite, operator $\mathcal{D}$ is elliptic, and the bilinear form $a$ is coercive.  The coercivity property (\ref{coercivity}) implies (\ref{garding}) with $C_2 = 0$.
\end{proof}

Let $L^2 \left( 0, T; X \right)$ and $H^1 \left( 0, T; X \right)$ denote Bochner spaces and $H^{-1} \left( \Omega \right)$ denote the dual space of $H_0^1 \left( \Omega \right)$. Let $u_0 \in L^2 \left( \Omega \right)$ be the transformed payoff function. The variational formulation of the equation (\ref{eq_bounded_domain_2}) reads as:  Find $u \in L^2 \left( 0, T; H_0^1 \left( \Omega \right) \right) \cap H^1 \left( 0, T; H^{-1} \left( \Omega \right) \right)$ such that
\begin{equation} \label{variational_formulation_1}
\left\langle \frac{\partial U }{\partial t }, v \right\rangle
+ a \left( U, v \right) = 0, \quad \forall v \in V, \,  \text{a.e. in} \left( 0, T \right); \, U \left( \cdot, 0 \right) = u_0.
\end{equation} 

\begin{thm}
There exists a unique solution $u$ to the problem~(\ref{variational_formulation_1}). 
\end{thm}

\begin{proof}
It is well-known that the existence and uniqueness of the solution follows directly from the G\aa rding inequality and continuity of the bilinear form~$a$ shown in Lemma~\ref{Lemma_Garding}, as detailed \cite{Achdou2005, Hilber2013, Kestler2013}.
\end{proof}

Let $\Psi$ be an orthonormal wavelet basis for the space $L^2 \left( \Omega \right)$ constructed using the tensor product from a univariate wavelet basis satisfying assumptions $A1)-A6)$, and consider its subset 
\begin{equation}
\Psi^k = \left\{ \psi_{\lambda} \in \Psi : \lambda \in \mathcal{J}^k \right\}, \quad  \mathcal{J}^k = \left\{ \lambda: \left\| \lambda \right\|_1 \leq k \right\}, 
\end{equation}
where $\left\| \cdot \right\|_1$ is the $l^1$ norm. This means that the sets $\Psi^k$ generate sparse tensor product spaces similarly as in Section~\ref{sparse_tensor_product_wavelets}. Let $X_k = {\rm span} \, \Psi^k$ and denote the dual of $X_k$ as~$X_k'$. 

The wavelet-Galerkin method consists of finding $U^k \in L^2 \left( 0, T; X_k \right) \cap H^1 \left( 0, T; X_k' \right)$ such that for all $v_k \in X_k$ and a.e. in $\left(0, T \right)$,
\begin{equation} \label{variational_formulation}
\left\langle \frac{\partial U_k}{ \partial t}, v_k \right\rangle
+ a \left( U_k, v_k \right) = 0,
\quad U_k \left( \cdot, 0 \right) = u_0. 
\end{equation} 

\begin{thm} 
There exists a unique solution to the problem~(\ref{variational_formulation}).
\end{thm}

\begin{proof}
This theorem also follows from the continuity of the bilinear form $a$ and the G\aa rding inequality shown in Lemma~\ref{Lemma_Garding}.
\end{proof}

Now, the Crank--Nicolson scheme and Rannacher time-stepping are applied. Let $M \in \mathbb{N}$ be a number of time steps, $\tau = T/M$ be the step size, $t_l=l \tau$, $l=0, \ldots, M$, be a time level and $U_k^l \left( x_1, \ldots, x_d \right) =  U_k \left( x_1, \ldots, x_d, t_l\right)$ be the solution at the $l$-th time level. The initial function $u_0$ representing the transformed payoff function is typically not smooth. It is known that using the Crank--Nicolson scheme with non-smooth initial function  can lead to spurious oscillations and reduction of the accuracy of the scheme. This can be avoided using  Rannacher time-stepping, which consists of replacing two Crank--Nicolson steps with four steps of the implicit Euler scheme with half step size \cite{Rannacher1984}.

The resulting scheme has the following form. For $l=0, 0.5, 1, 1.5$, use the implicit-Euler scheme with step $\tau / 2$, 
\begin{equation} \label{scheme_Rannacher}
\frac{ \left\langle U_k^{l+1/2}, v_k \right\rangle }{\tau /2}- \frac{ \left\langle U_k^{l}, v_k \right\rangle }{\tau /2 } + a \left(  U_k^{l+1/2} , v_k \right) = 0, \quad \forall v_k \in X_k.
\end{equation}
Then, for $l=2, \ldots, M-1$, use the Crank--Nicolson scheme
\begin{equation} \label{scheme_CN}
\frac{ \left\langle U_k^{l+1}, v_k \right\rangle }{\tau}- \frac{ \left\langle U_k^{l}, v_k \right\rangle }{\tau} + \frac{ a \left(  U_k^{l+1} , v_k \right)}{2} + \frac{ a \left( U_k^{l}, v_k \right)}{2}= 0, \quad \forall v_k \in X_k.
\end{equation}

In the following, we focus only on the Crank--Nicolson scheme, which forms the main part of the method. The study of the implicit Euler scheme is similar. Inserting $v_k = \psi_{\mu}$ and using a series representation
\begin{equation}
U_k^l = \sum_{ \psi_{ \lambda } \in \Psi^k } \left( \mathbf{c}_k^l \right)_{ \lambda } \psi_{ \lambda }  
\end{equation}
into (\ref{scheme_CN}) yields for $l=2, \ldots, M-1,$ a linear system
\begin{equation} \label{slar_Ak}
\mathbf{A}^k \mathbf{c}_k^{l+1} = \mathbf{f}_k^l, 
\end{equation}
where for $\psi_{\lambda}, \psi_{ \mu} \in \Psi^k$,
\begin{equation}
\mathbf{A}^k_{ \mu, \lambda } = \frac{ \left\langle \psi_{ \lambda },  \psi_{ \mu}\right\rangle }{ \tau } + \frac{ a \left(  \psi_{ \lambda },  \psi_{ \mu} \right) }{2}, \quad
\left( \mathbf{f}_k^l \right)_{ \mu} = \frac{ \left\langle U_k^{l}, \psi_{\mu} \right\rangle }{\tau} - \frac{ a \left( U_k^{l}, \psi_{\mu} \right)}{2}.
\end{equation}
The elements of matrix $\mathbf{A}^k$ can also be expressed as 
\begin{equation}
   \mathbf{A}^k_{ \mu, \lambda } = \left( \frac{1}{\tau} + \frac{r}{2} \right) \left\langle \psi_{ \lambda },  \psi_{ \mu}\right\rangle  + \sum_{i=1}^d \sum_{j=1}^d \frac{ P_{i,j} }{2} \left\langle \frac{ \partial \psi_{\lambda} }{\partial x_i}, \frac{ \partial \psi_{\mu} }{\partial x_j} \right\rangle. 
\end{equation} 
Due to the orthogonality of functions $\psi_{\lambda_i}$ and $\psi_{\mu_i}$, we have $\left\langle \psi_{ \lambda_i },  \psi_{ \mu_j }\right\rangle = \delta_{ \lambda_i, \mu_j }$. Hence, we obtain
\begin{eqnarray} \label{inner_product_1}
\left\langle \frac{ \partial \psi_{\lambda} }{\partial x_i}, \frac{ \partial \psi_{\mu} }{\partial x_j} \right\rangle &=& 
\left\langle \frac{ \partial }{\partial x_i} \prod_{k=1}^d \psi_{\lambda_k} \left( x_k \right), \frac{ \partial }{ \partial x_j } \prod_{l=1}^d \psi_{\mu_l} \left( x_l \right) \right\rangle \\
\nonumber  &=&  \left\langle \psi'_{ \lambda_i },  \psi_{ \mu_i} \right\rangle \left\langle \psi_{ \lambda_j },  \psi'_{ \mu_j}\right\rangle \prod_{\substack{ k=1 \\ k\neq i,j } }^d \delta_{ \lambda_k, \mu_k },
\end{eqnarray}
for $i \neq j$, and similarly
\begin{eqnarray} \label{inner_product_2}
\left\langle \frac{ \partial \psi_{\lambda} }{\partial x_i}, \frac{ \partial \psi_{\mu} }{\partial x_i} \right\rangle &=& 
\left\langle \frac{ \partial }{ \partial x_i } \prod_{k=1}^d \psi_{\lambda_k} \left( x_k \right) , \frac{ \partial }{\partial x_i} \prod_{l=1}^d \psi_{\mu_l} \left( x_l \right) \right\rangle \\
\nonumber  &=&  \left\langle \psi'_{ \lambda_i },  \psi'_{ \mu_i} \right\rangle
\prod_{ \substack{k=1 \\ k\neq i}}^d \delta_{ \lambda_k, \mu_k }.
\end{eqnarray}
Therefore, to compute the entries of the matrix $\mathbf{A}^k$, it is necessary to first compute the products $\left\langle \psi_{ \lambda_i },  \psi_{ \mu_i} \right\rangle$, $\left\langle \psi'_{ \lambda_i },  \psi_{ \mu_i} \right\rangle$, and $\left\langle \psi'_{ \lambda_i },  \psi'_{ \mu_i} \right\rangle$.

We consider matrices of these products. Let $\mathbf{M}^k$ and $\mathbf{B}^k$ be matrices with elements
\begin{equation}
\mathbf{M}^k_{\lambda,\mu} = \left\langle \psi'_{ \lambda },  \psi'_{ \mu } \right\rangle, \quad 
\mathbf{B}^k_{\lambda,\mu} = \left\langle \psi'_{ \lambda },  \psi_{ \mu } \right\rangle,
\quad \left| \lambda \right|, \left| \mu \right| \leq k.    
\end{equation}

Due to the orthogonality property of the basis used, the matrix 
$\mathbf{I}^k$ with elements 
\begin{equation}
    \mathbf{I}^k_{\lambda,\mu} = \left\langle \psi_{ \lambda },  \psi_{ \mu } \right\rangle,
    \quad \left| \lambda \right|, \left| \mu \right| \leq k, 
\end{equation}
is an identity matrix.

The next lemma shows that the condition numbers of matrices $\mathbf{M}^k$ and $\mathbf{B}^k$ increase exponentially with the parameter $k$ characterizing the level of a univariate basis. Let $\left\| \cdot \right\|$ denote the spectral matrix norm and ${\rm cond}$ denote the condition number of a matrix. 

\begin{lem} \label{ lemma_norms_MB }
There exists a constant $C \in \mathbb{R}$ such that
\begin{equation}
 \left\|  \mathbf{M}^k \right\| \leq C 2^{2k},  \quad 
 {\rm cond } \, \mathbf{M}^k \leq C 2^{2k}, \quad
    \left\| \mathbf{B}^k \right\| \leq C 2^{k}.
\end{equation}
\end{lem}

\begin{proof}
By assumption $A6)$, the set $\Psi$, when normalized in the $H^1$-norm, is a Riesz basis of $H_0^1 \left( I \right)$. Therefore, there exist constants $A_1, A_2 > 0$ such that
\begin{equation} \label{Riesz_H1}
 A_1^2  \sum_{ \lambda \in \mathcal{J}^k } d_{ \lambda }^2  \leq \left\|  \sum_{ \lambda \in \mathcal{J}^k }  d_{\lambda} \frac{ \psi_{\lambda} }{ \left\| \psi_{\lambda} \right\|_{H^1} } \right\|^2_1 \leq A_2^2  \sum_{ \lambda \in \mathcal{J}^k } d_{\lambda}^2 .
\end{equation}
By (\ref{psi_jk_translations_dilations}), we have $d \, 2^{ \left| \lambda \right| } \leq \left\| \psi_{ \lambda } \right\|_1 \leq D \, 2^{ \left| \lambda \right| }$.
Hence, for $d_{ \lambda } = 2^{  \left| \lambda \right|
} \, c_{ \lambda }$, we obtain
\begin{equation} 
 B_1^2  \sum_{ \lambda \in \mathcal{J}^k } 2^{2 \left| \lambda \right| } c_{ \lambda }^2 \leq \left\|  \sum_{ \lambda \in \mathcal{J}^k } c_{  \lambda } \psi_{ \lambda } \right\|^2_1 \leq B_2^2  \sum_{ \lambda \in \mathcal{J}^k } 2^{2 \left| \lambda \right| } c_{  \lambda }^2
\end{equation}
for some constants $B_1, B_2 > 0$. Because the $H^1$-norm and the $H^1$-seminorm are equivalent in $H_0^1 \left( I \right)$, the relation (\ref{Riesz_H1}) remains valid only with different constants when replacing the norm $\left\| \cdot \right\|_1$ with the seminorm $\left| \cdot \right|_1$, 
\begin{equation} \label{Riesz_seminorm_H1}
 C_1^2  \sum_{ \lambda \in \mathcal{J}^k } 2^{2 \left| \lambda \right| } c_{\lambda}^2 \leq \left|  \sum_{ \lambda \in \mathcal{J}^k }  c_{\lambda} \psi_{\lambda} \right|^2_1 \leq C^2_2  \sum_{ \lambda \in \mathcal{J}^k } 2^{2 \left| \lambda \right| } c_{\lambda}^2.
\end{equation}

This implies that
\begin{eqnarray} \label{estimate_H1}
C_1^2  \sum_{ \lambda \in \mathcal{J}^k } c_{ \lambda }^2 &\leq&
C_1^2  \sum_{ \lambda \in \mathcal{J}^k } 2^{2 \left| \lambda \right| } c_{ \lambda }^2 \leq \left| \sum_{ \lambda \in \mathcal{J}^k } c_{ \lambda } \psi_{j,l} \right|_1^2 \\
\nonumber &\leq& C_2^2  \sum_{ \lambda \in \mathcal{J}^k } 2^{2 \left| \lambda \right| } c_{ \lambda }^2 \leq C_2^2  \, 2^{2k} \sum_{ \lambda \in \mathcal{J}^k } c_{ \lambda }^2.
\end{eqnarray}

Relation (\ref{estimate_H1}) can be rewritten to matrix form
\begin{equation} \label{estimate_Mk}
C_1^2 \left\| \mathbf{c}  \right\|^2 \leq \left( \mathbf{c} \right)^T \mathbf{M}^k \mathbf{c} \leq C_2^2 \, 2^{2k} \left\| \mathbf{c}  \right\|^2.
\end{equation}
Because matrix $\mathbf{M}^k$ is symmetric, using the Rayleigh quotient for $\mathbf{M}^k$ and (\ref{estimate_Mk}) imply that the maximal and minimal eigenvalues of $\mathbf{M}$ satisfy
\begin{equation}
    \lambda_{min} \left( \mathbf{M}^k \right) \geq C_1^2, \quad \left\|  \mathbf{M}^k \right\| = \lambda_{max} \left( \mathbf{M}^k \right) \leq C_2^2 2^{2k},
\end{equation}
and thus $ {\rm cond } \, \mathbf{M}^k \leq C_2^2 \, 2^{2k} / C_1^2$.

Similarly, using Cauchy--Schwarz inequality,
\begin{eqnarray} 
\left| \left\langle \sum\limits_{j \in \mathcal{J}^k }  c_{\lambda} \psi_{\lambda}, \sum\limits_{j \in \mathcal{J}^k }  c_{\lambda} \psi'_{\lambda} \right\rangle \right|
&\leq& \left\| \sum\limits_{j \in \mathcal{J}^k }  c_{\lambda} \psi_{\lambda} \right\|  \left| \sum\limits_{j \in \mathcal{J}^k }  c_{\lambda} \psi_{\lambda} \right|_1 \\
\nonumber &\leq& C_2  \left( \sum\limits_{j \in \mathcal{J}^k } c_{\lambda}^2 \right)^{1/2}  \left( \sum\limits_{j \in \mathcal{J}^k } 2^{2 \left| \lambda \right| } c_{\lambda}^2 \right)^{1/2} \\
\nonumber &\leq&  C_2 \, 2^{k} \sum\limits_{j \in \mathcal{J}^k } c_{ \lambda }^2.
\end{eqnarray}

This is equivalent to
\begin{equation}
\left( \mathbf{c} \right)^T \mathbf{B}^k \mathbf{c} \leq C_2 \, 2^{k} \left\| \mathbf{c}  \right\|^2.
\end{equation}

In contrast to the matrix $\mathbf{M}^k$, the matrix  $\mathbf{B}^k$ is not symmetric. Therefore, the estimate of its spectral norm using eigenvalues cannot be used here as above. However, it is known that
\begin{equation} \label{norm_numerical_range}
    \left\| \mathbf{N} \right\| \leq 2 \sup\limits_{0 \neq \mathbf{d} }
    \frac{\mathbf{d}^T \mathbf{N} \mathbf{d}}{ \left\| \mathbf{d} \right\|^2},
\end{equation}
holds for any square matrix $\mathbf{N}$; see, e.g., Theorem 1.3-1 in \cite{Gustafson1997}. This implies that, indeed, $\left\| \mathbf{B}^k \right\| \leq C 2^k $, $C = 2 \, C_2$.
\end{proof}

To analyze matrices $\mathbf{A}^k$ corresponding to bases constructed by sparse tensor product, let us first consider discretization matrices $\mathbf{A}^k_{ani}$ corresponding to bases constructed using anisotropic approach. By (\ref{inner_product_1}) and (\ref{inner_product_2}), these matrices can be expressed as the sum of the terms, which are tensor products of matrices $\mathbf{I}^k$, $\mathbf{M}^k$, and $\mathbf{B}^k$, 
\begin{eqnarray} \label{Ak_ani_tensor}
\mathbf{A}^k_{ani} &=&  \left( \frac{1}{\tau} + \frac{r}{2} \right) \mathbf{I}^k + \sum_{i=1}^d \frac{P_{i,i}}{2} \left( \bigotimes\limits_{j=1}^{i-1} \mathbf{I}^k \right) \otimes \mathbf{M}^k \otimes \left( \bigotimes\limits_{j=i+1}^d \mathbf{I}^k \right) \\
\nonumber &-& \sum_{i=1}^{d-1} \sum_{j=i+1}^d P_{i,j}
\left( \bigotimes\limits_{k=1}^{i-1} \mathbf{I}^k \right) \otimes \mathbf{B}^k \otimes \left( \bigotimes\limits_{k=i+1}^{j-1} \mathbf{I}^k \right) \otimes \mathbf{B}^k \otimes \left( \bigotimes\limits_{k=j+1}^{d} \mathbf{I}^k \right).
\end{eqnarray}

This is not valid for sparse tensor product bases. However, matrices $\mathbf{A}^k$ can be divided into blocks such that each block can be computed using the tensor product of the matrices corresponding to one-dimensional bases similarly as in (\ref{Ak_ani_tensor}).

More precisely, for $m = \left( m_1, \ldots, m_d \right)$ and $n = \left( n_1, \ldots, n_d \right)$ such that $\left\| m \right\|_1 \leq k$ and $\left\| n \right\|_1 \leq k$, let $\mathbf{A}^{m,n}$ be a block of $\mathbf{A}^k$ with elements 
\begin{equation}
\mathbf{A}^{m,n}_{\lambda, \mu} =   \left( \frac{1}{\tau} + \frac{r}{2} \right) \left\langle \psi_{ \lambda },  \psi_{ \mu}\right\rangle  + \sum_{i=1}^d \sum_{j=1}^d \frac{ P_{i,j} }{2} \left\langle \frac{ \partial \psi_{\lambda} }{\partial x_i}, \frac{ \partial \psi_{\mu} }{\partial x_j} \right\rangle, 
\end{equation}
where $\lambda = \left( \lambda_1, \ldots, \lambda_d \right)$ and 
$\mu = \left( \mu_1, \ldots, \mu_d \right)$ are such that
$\left| \lambda_i \right| = m_i$ and $\left| \mu_i \right| = n_i$ for $i=1, \ldots d$.
This means that $m$ and $n$ characterize levels of univariate functions $\psi_{\lambda_i}$ and $\psi_{\mu_i}$, which form basis functions $\psi_{\lambda}$ and $\psi_{\mu}$ via the tensor product. Thus, the matrix $\mathbf{A}^k$ can be written as
\begin{equation}
    \mathbf{A}^k = \left( \mathbf{A}^{m,n} \right)_{ 0 \leq \left\| m \right\|_1, \left\| n \right\|_1 \leq k }.
\end{equation}

Similarly to matrix $\mathbf{A}^k_{ani}$, each block $\mathbf{A}^{m,n}$ can be computed as
\begin{eqnarray} \label{Amn_tensor_product}
\mathbf{A}^{m,n} &=&  \left( \frac{1}{\tau} + \frac{r}{2} \right) \mathbf{I} + \sum_{i=1}^d \frac{P_{i,i}}{2} \bigotimes\limits_{j=1}^{i-1} \mathbf{I} \otimes \mathbf{M}^{m_i,n_i} \otimes \bigotimes\limits_{j=i+1}^d \mathbf{I} \\
\nonumber &-& \sum_{i=1}^{d-1} \sum_{j=i+1}^d P_{i,j}
 \bigotimes\limits_{k=1}^{i-1} \mathbf{I} \otimes \mathbf{B}^{m_i,n_i} 
\otimes \bigotimes\limits_{k=i+1}^{j-1} \mathbf{I} \otimes \mathbf{B}^{m_j,n_j} \otimes  \bigotimes\limits_{k=j+1}^{d} \mathbf{I}.
\end{eqnarray}
Here, $\mathbf{I}$ is an identity matrix of appropriate size and $\mathbf{M}^{m,n}$ and $\mathbf{B}^{m,n}$ are matrices with elements 
\begin{equation}
\mathbf{M}^{m,n}_{\lambda,\mu} = \left\langle \psi'_{\lambda}, \psi'_{\mu} \right\rangle, 
\quad \mathbf{B}^{m,n}_{\lambda, \mu} = \left\langle \psi'_{\lambda}, \psi_{\mu} \right\rangle, \quad
\left| \lambda \right| = m, \quad \left| \mu \right| = n.
\end{equation}
The fact that many matrices in (\ref{Amn_tensor_product}) are identity matrices greatly simplifies and streamlines numerical computations with these matrices in comparison with other non-orthogonal tensor product bases.

Due to the orthogonality of the basis and the block tensor product structure (\ref{Amn_tensor_product}), the matrices $\mathbf{A}^{k}$ are uniformly conditioned even without preconditioning, which is studied in the following theorem.

\begin{thm} Let $C_2$ be an upper Riesz bound with respect to the $H^1$ seminorm given by (\ref{Riesz_seminorm_H1}). If 
\begin{equation} \label{assumption_step_size}
\tau < \frac{ \gamma }{ A \, 4^{k} }, \quad A = 2 \, C_2^2 \sum\limits_{i=1}^d \sum\limits_{j=1}^d P_{i,j}, \quad \gamma \in \left( 0, 1 \right),
\end{equation}
then there exists a constant $C$  independent of $k$ such that $ {\rm cond} \, \mathbf{A}^k \leq C$. 
\end{thm}

\begin{proof}
The matrix $\mathbf{A}^k$ is a submatrix of the matrix $\mathbf{A}^k_{ani}$ and both matrices are symmetric positive definite. Therefore, $ {\rm cond } \, \mathbf{A}^k \leq {\rm cond } \, \mathbf{A}^k_{ani}$ and it is sufficient to show the uniform boundedness of condition numbers for $\mathbf{A}^k_{ani}$.

For the sake of simplicity, let denote
\begin{eqnarray}
\mathbf{C}^k &=& \sum_{i=1}^d \frac{P_{i,i}}{2} \left( \bigotimes\limits_{j=1}^{i-1} \mathbf{I} \right) \otimes \mathbf{M}^k \otimes \left( \bigotimes\limits_{j=i+1}^d \mathbf{I} \right) \\
\nonumber &-& \sum_{i=1}^{d-1} \sum_{j=i+1}^d P_{i,j}
\left( \bigotimes\limits_{k=1}^{i-1} \mathbf{I} \right) \otimes \mathbf{B}^k \otimes \left( \bigotimes\limits_{k=i+1}^{j-1} \mathbf{I} \right) \otimes \mathbf{B}^k \otimes \left( \bigotimes\limits_{k=j+1}^{d} \mathbf{I} \right).
\end{eqnarray}
The proof of Theorem~\ref{ lemma_norms_MB } implies that 
\begin{eqnarray} \label{estimate_Ck}
\left\| \mathbf{C}^k \right\| \leq \sum_{i=1}^d \frac{P_{i,i}}{2} \left\| \mathbf{M}^k \right\| + \sum_{i=1}^{d-1} \sum_{j=i+1}^d P_{i,j} \left\| \mathbf{B}^k \right\|^2 \leq 2 \, C_2^2 \, 2^{2k} \sum_{ i=1}^d \sum_{j=1}^d P_{i,j}.
\end{eqnarray}
Thus, we have
\begin{equation}
\left\| \mathbf{A}^k_{ani} \right\| = \left\| \left( \frac{1}{\tau} + \frac{r}{2} \right) \mathbf{I} + \mathbf{C}^k \right\| 
\leq \left( \frac{1}{\tau} + \frac{r}{2} \right) + A \, 2^{2k }
\end{equation}
with $A$ given by (\ref{assumption_step_size}).

Now, we aim to estimate the norm of the inverse of $\mathbf{A}^k_{ani}$. Using
\begin{equation}
\left( \left( \frac{1}{\tau} + \frac{r}{2} \right) \mathbf{I} + \mathbf{C}^k \right) \left( \left( \frac{1}{\tau} + \frac{r}{2} \right) \mathbf{I} + \mathbf{C}^k \right)^{-1} = \mathbf{I},   
\end{equation}
we obtain
\begin{equation} \label{relation_for_inverse}
\left( \frac{1}{\tau} + \frac{r}{2} \right) \left( \left( \frac{1}{\tau} + \frac{r}{2} \right) \mathbf{I} + \mathbf{C}^k \right)^{-1} = \mathbf{I} -
\mathbf{C}^k \left( \left( \frac{1}{\tau} + \frac{r}{2} \right) \mathbf{I} + \mathbf{C}^k \right)^{-1}.   
\end{equation}
To simplify notations, denote 
\begin{equation}
\mathbf{F} = \left( \left( \frac{1}{\tau} + \frac{r}{2} \right) \mathbf{I} + \mathbf{C}^k \right)^{-1}. 
\end{equation}
Taking spectral norm of (\ref{relation_for_inverse}), we have
\begin{eqnarray}
\left( \frac{1}{\tau} + \frac{r}{2} \right) \left\| \mathbf{F} \right\| & \leq &
1 + \left\| \mathbf{C}^k \right\| \cdot \left\| \mathbf{F} \right\| \\
\left\| \mathbf{F} \right\| & \leq & \frac{ 1 }{ \left( \frac{1}{\tau} + \frac{r}{2} \right) - \left\| \mathbf{C}^k \right\| },
\end{eqnarray}
because the denominator is positive due to assumption (\ref{assumption_step_size}) and relation (\ref{estimate_Ck}). Hence,
\begin{equation}
\left\| \left( \mathbf{A}^k_{ani} \right)^{-1} \right\| \leq \frac{1}{ \left( \frac{1}{\tau} + \frac{r}{2} \right) - A 2^{2k} },
\end{equation}
and using assumption (\ref{assumption_step_size}), we have  
\begin{equation}
{ \rm cond}  \, \mathbf{A}^k \leq { \rm cond}  \, \mathbf{A}^k_{ani} \leq \frac{ 1 +  A \, 2^{2k} \left( \frac{1}{\tau} + \frac{r}{2} \right)^{-1}  }{   1 -  A \, 2^{2k} \left( \frac{1}{\tau} + \frac{r}{2} \right)^{-1} } \leq \frac{ 1 + \gamma }{ 1 - \gamma }.
\end{equation}
\end{proof}

Note that the condition number of $\mathbf{A}^k$ is not only uniformly bounded, but also that the upper bound does not depend on dimension $d$, which is the significant result because, for many methods, the condition number grows exponentially with~$d$.
Note that condition (\ref{assumption_step_size}) is not very restrictive. Indeed, if spline wavelet basis of order $p$ is used, then by Theorem~\ref{thm_approx_sparse} the optimal approximation order with respect to the $L^2$ norm is $\mathcal{O} \left( 2^{-pk} \right)$, if we neglect the term $k^{\left( d-1 \right)/2 }$. The optimal $L^2$ error for the Crank--Nicolson scheme is $\mathcal{O} \left( \tau^2 \right)$. Therefore, the optimal choice is $\tau = C 2^{ -pk / 2 } \leq C 4^{-k}$ for $p \geq 4$.

Let $\mathbf{c}^*$ be the exact solution of equation (\ref{slar_Ak}) and $\mathbf{c}^k$ be an approximate solution after $k$ conjugate gradient iterations. It is well known that the speed of convergence of $\mathbf{c}^k$ depends on the condition number of the matrix $\mathbf{A}^k$. More precisely, 
\begin{equation*}
\left\|\mathbf{c}^k - \mathbf{c}^* \right\|_A
\leq 2 \left( \frac{\sqrt{ { \rm cond} \, \mathbf{A}^k }-1}{\sqrt{ { \rm cond} \, \mathbf{A}^k }+1}  \right)^k \left\|\mathbf{c}^0 - \mathbf{c}^{*} \right\|_A,
\end{equation*}
where $\left\|\cdot\right\|_A$ is given by $\left\|\mathbf{x}\right\|_A=\sqrt{\mathbf{x}^T \mathbf{A} \mathbf{x} }$.

Hence, the uniformly bounded condition number of matrix $\mathbf{A}^k$ implies a uniformly bounded number of conjugate gradient iterations, if the same stopping criterion for relative residuals is used for all $k$. Moreover, the upper bound for a number of iterations does not increase with the dimension $d$.

\bigskip
\section{Numerical examples}
\label{Section_examples}

In this section, numerical examples are presented to confirm the theoretical results and illustrate the efficiency and reliability of the method. The proposed wavelet scheme is used to evaluate prices of European-style options on the geometric average of the underlying assets. The advantage of these types of options is that after a transformation, the option contract can be valuated using the standard one-dimensional Black--Scholes model \cite{Berridge2004, Leentvaar2008}. Then, the resulting errors can be computed comparing numerical and analytic solutions. 

The geometric average of asset prices $S_1, \ldots, S_d$ is defined as
\begin{equation} \label{geom_average_assets}
    S = \left( \prod\limits_{i=1}^d S_i \right)^{1/d}.
\end{equation}

The payoff function representing the value of the option at maturity is for the strike $K$ given by
\begin{equation}
    V \left( S_1, \ldots, S_d, 0 \right) = \max \left( K - S, 0 \right)
\end{equation}
for a put option and 
\begin{equation}
    V \left( S_1, \ldots, S_d, 0 \right) = \max \left( S - K, 0 \right)
\end{equation}
for a call option. The graphs of payoff functions for $K=10$ and $d=2$ are displayed in Figure~\ref{Fig9}.

\begin{figure}[htbp]
\centering
\caption{ Payoff functions for a put option (left) and a call option (right) on the geometric average of two assets with the strike $K=10$. }
\includegraphics[width=13.0cm]{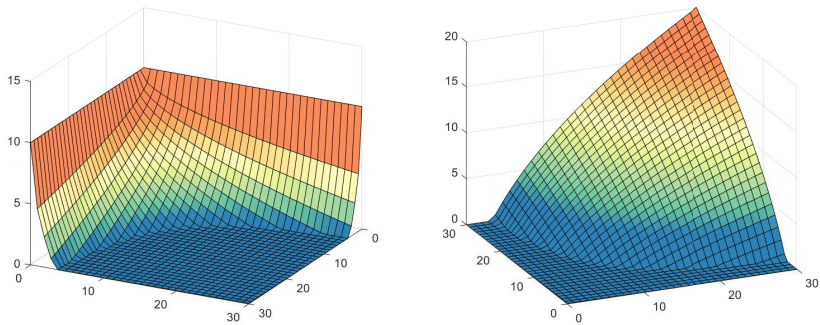} 
\label{Fig9}
\end{figure} 

The price of a put option on the geometric average of $d$ assets can be computed using the standard Black--Scholes formula
\begin{equation*} \label{BS_formula}
V_{put} \left( S_1, \ldots, S_d, t \right) = K e^{-rt} N \left( - d_2 \right) - S e^{ - \delta t} N \left( - d_1 \right),
\end{equation*}
where $S$ is given by (\ref{geom_average_assets}), $N$ is the standard normal density function 
\begin{equation}
    N \left( x \right) = \frac{ 1 }{ \sqrt{ 2 \pi } } \int\limits_{ - \infty}^x
    e^{ - \frac{ t^2 }{ 2 } } dt,
\end{equation}
and
\begin{equation}
 d_1 =  \frac{ \ln \frac{ S }{  K } + \left( r - \delta + \frac{ \sigma^2 }{2} \right) t}{ \sigma \sqrt{t} }, \quad 
 d_2 =  \frac{ \ln \frac{ S }{ K } + \left( r - \delta - \frac{ \sigma^2 }{2} \right) t}{ \sigma \sqrt{t} }.
\end{equation}
The parameters $\sigma$ and $\delta$ are given by
\begin{equation}
\sigma^2 = \frac{1}{d^2} \sum_{i=1}^d \sum_{j=1}^d \rho_{ij} \sigma_i \sigma_j, \quad \delta = \frac{ 1 }{2d } \sum\limits_{i=1}^d  \sigma_i^2  - \frac{ \sigma^2 }{2}. 
\end{equation}
The price of a call option can then be computed using the put-call parity
\begin{equation}
V_{call} \left( S_1, \ldots, S_d, t \right) = V_{put} \left( S_1, \ldots, S_d, t \right)  + S e^{ - \delta t}  - K e^{-rt}.
\end{equation}

{\noindent \bf Example. Options on the geometric average of $d$ assets}

In this example, numerical results are presented for options on the geometric average with parameters taken from \cite{Berridge2004}. The strike price is $K=10$ and the option is maturing in $T=1$ year. The correlation coefficients are $ \rho_{i,j} = 0.25$ for $i \neq j$, the risk-free interest rate for a given horizon is $r = 0.06$ p.a., and the corresponding volatilities are $\sigma_i = 0.2$. The parameters characterizing the domain for the computation are set to  $S^{min}=0.1$ and $S^{max}=50$. Figure~\ref{Fig10} shows the resulting functions representing prices of put and call options for these parameters computed using the proposed method with constructed orthogonal cubic spline wavelet basis.  Because artificial boundary conditions are used, the plots are shown only in the region $\left( 1, 30 \right)^2$, to avoid the area near $S_1=0$ and $S_2=0$.

\begin{figure}[htbp]
\centering
\caption{\emph{ Functions representing values of a put option (left) and a call option (right) on the geometric average of two assets.   }}
\includegraphics[width=13.0cm]{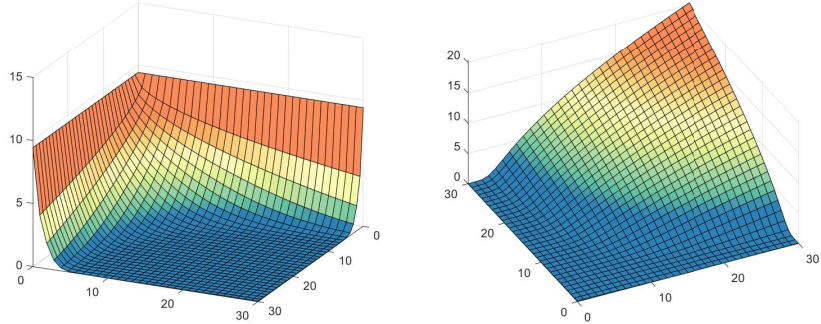} 
\label{Fig10}
\end{figure} 

Table~\ref{Tab1} presents the results. Parameter $d$ denotes the number of underlying assets, i.e., the spatial dimension of the problem. The parameter $k$ is a number characterizing basis $\Psi^k$ constructed using sparse tensor product approach, $N$ is the number of basis functions, and $M$ is the number of time steps. The resulting system is solved by the conjugate gradient method without preconditioning. The iterations stop when the relative residual is smaller than $10^{-10}$ and the resulting number of iterations is denoted $it$. The pointwise errors are computed for prices of underlying assets $P_1 = \left( K/2, \ldots, K/2 \right) \in \mathbb{R}^d$, $P_2 = \left( K, \ldots, K \right) \in \mathbb{R}^d$,
and $P_3 = \left( 3K/2, \ldots, 3K/2 \right) \in \mathbb{R}^d$ 
at time to maturity $t = T$, 
\begin{equation}
    e_k \left( P_i \right) = \left| V_k \left( P_i, T \right) -
    V \left( P_i, T \right) \right|, \quad i=1,2,3,
\end{equation}
where $V$ represents exact value computed by analytic formula (\ref{BS_formula}) and $V_k$ stands for the solution computed using the proposed method with basis $\Psi^k$.

\begin{table}[ht]
\centering
\caption{ Pointwise errors $e_k \left( P_i \right)$, $i=1,2,3$, and numbers of iterations for options on the geometric average for $d$ assets. }
\label{Tab1}
\begin{tabular}
{ | r  r r  r | c c  c | c c c  |}
    \hline
  &  &  &   &    \multicolumn{3}{|c|}{put option }  &  \multicolumn{3}{c|}{ call option  }   \\
  & $k$ &  $N$ & $M$ & $it$ & $ e_k \left( P_1 \right) $ & $ e_k \left( P_2 \right) $ & $it$ & $  e_k \left( P_2 \right) $ & $ e_k \left( P_3 \right) $ \\
  \hline \hline
$d=2$ & 0 &     36 &     1 & 9 & 4.09e-1 & 3.89e-1 &  8 & 4.67e-1 & 1.98e-1 \\
      & 1 &    144 &     4 & 9 & 2.16e-3 & 5.29e-3 &  9 & 2.14e-2 & 6.74e-2 \\
      & 2 &    432 &    16 & 8 & 1.77e-3 & 1.89e-3 &  8 & 6.58e-3 & 6.69e-3 \\
      & 3 &  1 152 &    64 & 7 & 6.42e-4 & 9.11e-4 &  8 & 8.85e-4 & 1.95e-3 \\
      & 4 &  2 880 &   256 & 6 & 7.60e-5 & 7.31e-5 &  6 & 6.93e-5 & 1.11e-5 \\
      & 5 &  6 912 & 1 024 & 5 & 4.51e-6 & 3.52e-7 &  5 & 5.84e-6 & 6.56e-6 \\
      & 6 & 16 128 & 4 096 & 6 & 2.28e-7 & 2.81e-7 &  5 & 2.28e-7 & 7.83e-7 \\
\hline 
$d=3$ & 0 &    216 &     1 & 9 & 3.62e-1 & 1.83e-1 &  9 & 3.26e-1 & 3.04e-1 \\
      & 1 &  1 728 &     4 & 9 & 1.07e-3 & 1.12e-2 &  9 & 3.52e-2 & 6.81e-2 \\
      & 2 &  6 912 &    16 & 8 & 9.72e-4 & 6.45e-4 &  8 & 6.63e-3 & 6.77e-3 \\
      & 3 & 22 464 &    64 & 6 & 5.38e-5 & 1.13e-3 &  6 & 1.37e-3 & 2.36e-3 \\
      & 4 & 65 664 &   256 & 5 & 3.37e-6 & 1.98e-4 &  5 & 1.93e-4 & 1.56e-4 \\
      & 5 &179 712 & 1 024 & 4 & 9.97e-6 & 8.25e-6 &  4 & 8.36e-6 & 3.62e-6 \\
\hline
$d=4$ & 0 &     1 296 &   1 & 10 & 3.36e-1 & 1.69e-1 & 10 & 2.34e-1 & 4.18e-1 \\
      & 1 &    20 736 &   4 & 10 & 7.32e-3 & 8.71e-3 & 10 & 3.96e-2 & 7.63e-2 \\
      & 2 &   103 680 &  16 &  8 & 3.69e-4 & 7.60e-4 &  8 & 8.05e-3 & 5.12e-3 \\
      & 3 &   393 984 &  64 &  6 & 1.17e-4 & 2.57e-4 &  6 & 3.71e-4 & 2.19e-3 \\
\hline
$d = 5$& 0 &     7 776 &  1 & 10 & 3.21e-1 & 1.58e-1 & 10 & 1.63e-1 & 5.05e-1  \\
      & 1 &   248 832 &  4 & 10 & 1.66e-3 & 6.16e-3 & 10 & 4.52e-2 & 8.45e-2 \\
      & 2 & 1 492 992 & 16 &  8 & 6.46e-4 & 2.65e-4 &  7 & 1.14e-2 & 2.53e-3 \\
 \hline
\end{tabular}
\end{table}

\medskip
\section{Conclusion}
In this paper, we presented a wavelet-based method for option pricing under the Black--Scholes model. Because the wavelet basis used is a cubic spline wavelet basis, the method is high-order accurate with respect to the variables representing prices, and the sparse tensor product structure of the basis enables us to overcome the so-called ``curse of dimensionality." Due to the orthogonality of the basis, the resulting stiffness matrices have uniformly bounded condition numbers even without preconditioning. Moreover, the condition number does not increase with the dimension. The numerical experiments for European-style options on the geometric average confirm the efficiency and applicability of the scheme; high-order convergence is achieved; and the numbers of iterations are small, uniformly bounded and not growing with the dimension. The described method can be used for pricing any option that is represented by the presented Black--Scholes model such as basket options, options on maximum or minimum of several underlying assets, and real options.

\section*{Acknowledgments}
This work was supported by grant No. GA22-17028S funded by the Czech Science Foundation and grant No. PURE-2020-4003 funded by the Technical University of Liberec.






\newpage
\appendix
\section{ Coefficients of constructed generators }
\label{Appendix_1}

\begin{longtable}{ | c | r r |  r |  r | r | r | }
\caption{ Coefficients of the scaling generators corresponding to the terms $x^3$, $x^2$, $x$, and $1$ on $\left[ a, b \right]$. } 
\label{Tab3} \\

\hline 
 & a & b  &   &  coefficient  &  &  coefficient \\ \hline
\endfirsthead

\multicolumn{6}{c}%
 {{\bfseries \tablename\ \thetable{} -- continued from previous page}} \\
\hline
 & a & b  &   &  coefficient  &  &  coefficient \\ \hline
\endhead

\hline
\endfoot

\hline
\endlastfoot

 \hline 
$\phi_1$ & 0 & 0.5 & $x^3$ & -26.25320493255574  & $x^2$ &  19.68990369941681 \\*
             &    &  & $x$ &                  0  &   $1$ &                 0  \\
         & 0.5 & 1 & $x^3$ &  26.25320493255574  & $x^2$ & -59.06971109825042 \\*
            &     &  & $x$ &  39.37980739883362  &   $1$ & -6.563301233138936 \\ \hline
$\phi_2$ & 0 & 0.5 & $x^3$ &  81.97560612767679  & $x^2$ & -40.98780306383839 \\*
          &     &    & $x$ &                  0  &   $1$ &                  0 \\
   &   0.5 &   1 &   $x^3$ &  81.97560612767679  & $x^2$ & -204.9390153191920 \\*
       &    &        & $x$ &  163.9512122553536  &   $1$ & -40.98780306383839 \\
 \hline
$\phi_3$ & 0 & 0.25& $x^3$ &  53.41388035788556  & $x^2$ & -35.17323440989752 \\*
        &     &      & $x$ &                  0  &   $1$ &                  0 \\
      & 0.25 & 0.5 & $x^3$ & -527.8691353655910  & $x^2$ &  604.3590514426924 \\*
         &    &      & $x$ & -210.7755774781431  &   $1$ &  21.80567362442823 \\
      & 0.5 & 0.75 & $x^3$ &  167.3619498129961  & $x^2$ & -332.4004596541814 \\*
       &    &        & $x$ &  204.5606197347904  &   $1$ & -38.57643285514344 \\
     & 0.75  &  1  & $x^3$ & -381.1201871105696  & $x^2$ &  1023.674050039347 \\*
          &    &     & $x$ & -903.9875387469848  &   $1$ &  261.4336758182076 \\
          \hline
$\phi_4$ & 0 & 0.25 &$x^3$ &  26.62362537753055  & $x^2$ &  15.35514686530018 \\*
         &     &     & $x$ &                  0  &   $1$ &                  0 \\
    & 0.25  &  0.5 & $x^3$ &  1269.300093977873  & $x^2$ & -1351.930373703732 \\*
      &      &      & $x$ &  450.6409224219518  &   $1$ & -46.62170539180379 \\
  &   0.5 &  0.75 & $x^3$ &  567.2864791328406  & $x^2$ & -1200.276995000695 \\*
          &    &    & $x$ &  825.4977548526893  &   $1$ & -184.2117644273027 \\
    &   0.75  & 1  & $x^3$ & -399.7440135001654  & $x^2$ &  1084.571643661303 \\*
             &   &   & $x$ & -969.9112468221101  &   $1$ &  285.0836166609723 \\
 \hline
$\phi_5$& 0 & 0.25 & $x^3$ &  10.19187443147913  & $x^2$ &  28.81827269560206 \\*
          &    &     & $x$ &  27.06092209676673  &   $1$ &  8.434523832643798 \\
     &  0.25 & 0.5 & $x^3$ &  122.4194193918863  & $x^2$ &  242.0020294473649 \\*
          &    &    & $x$ &  157.4525751037239  &   $1$ &  33.65839594516685 \\
  &   0.5 &  0.75 & $x^3$ &  179.9690473081142  & $x^2$ &  185.5152420484449 \\*
            &    &   & $x$ &  57.80356676763296  &   $1$ &  5.149292116379875 \\
    &  0.75  &  1  & $x^3$ & -365.9470852800144  & $x^2$ & -134.8104412073693 \\*
            &   &    & $x$ &                  0  &   $1$ &  2.188816056270510 \\
     & 0   &  0.25 & $x^3$ &  79.26943625448920  & $x^2$ & -58.48104390125948 \\*
          &    &    & $x$ &                  0  &   $1$ &  2.188816056270510 \\
    &  0.25  & 0.5 & $x^3$ & -354.3621713021780  & $x^2$ &  418.6216724392985 \\*
      &      &     & $x$   & -157.2454317534039  & $1$   &  18.45674809140953 \\
      & 0.5 & 0.75 & $x^3$ & -50.59040574008613  & $x^2$ &  103.1911403389635 \\*
      &     &      & $x$   & -69.64372382463778  & $1$   &  15.54205645684875 \\
      & 0.75 & 1   & $x^3$ & -2.187234289273272  & $x^2$ &  6.198136014518121 \\* 
             &  &  & $x$   & -5.834569161216426  & $1$   &  1.823667435971577 \\
\hline
$\phi_6$&0&0.25 & $x^3$ & -15.40091533093236 & $x^2$ & -43.54717325197819 \\*
            &  &  & $x$ & -40.89160051115930 &   $1$ & -12.74534259011347 \\
   & 0.25 & 0.5 & $x^3$ & -184.9074696818534 & $x^2$ & -365.5196063064310 \\*
   &      &     &   $x$ & -237.8079396256591 & $1$   & -50.83368094965349 \\
   & 0.5 & 0.75 & $x^3$ & -265.5803178284557 & $x^2$ & -271.7997826856661 \\*
   &     &      &   $x$ & -83.58347989494250 &   $1$ & -7.235513007811689 \\
   & 0.75 &   1 & $x^3$ &  989.1663297899967 & $x^2$ &  460.3081592797291 \\*
   &     &      &   $x$ &  47.20549465929524 &   $1$ & -0.689599373051132 \\
   &   0 & 0.25 & $x^3$ &  273.6635501100016 & $x^2$ & -249.4001766673951 \\*
   &     &      &   $x$ &  47.20549465929524 &   $1$ & -0.689599373051132 \\
   & 0.25&  0.5 & $x^3$ & -610.5427721702383 & $x^2$ &  723.4466794016099 \\*
   &     &      &   $x$ & -273.4292479476622 &   $1$ &  32.48188156000418 \\
   & 0.5 & 0.75 & $x^3$ & -98.42902348889925 & $x^2$ &  199.2567948682677 \\*
   &     &      & $x$   & -133.3246749253244 &   $1$ &  29.46284759700341 \\
   &0.75 &    1 & $x^3$ & -5.217694467034758 & $x^2$ &  14.77336749684687 \\*
   &     &      & $x$   & -13.89365159258947 &   $1$ &  4.337978562777356 \\ \hline
$\phi_L$&0&0.25 & $x^3$ & -328.4179396059963 & $x^2$ &  294.5324087893769 \\*
   &     &      &   $x$ & -51.91281405451471 &   $1$ &                  0 \\
   &0.25 &  0.5 & $x^3$ &  794.2029204485535 & $x^2$ & -940.6297684522534 \\*
   &     &      &   $x$ &  355.1768633060724 &   $1$ & -42.11573420089721 \\
   & 0.5 & 0.75 & $x^3$ &  125.7725578721856 & $x^2$ & -254.8795911518400 \\*
   &     &      &   $x$ &  170.7494579379349 &   $1$ & -37.78578051988584 \\
   & 0.75&    1 & $x^3$ &  6.495819573712430 & $x^2$ & -18.39405161997127 \\*
   &     &      &   $x$ &  17.30064451880524 &   $1$ & -5.402412472546407 \\ \hline
$\phi_R$&0&0.25 & $x^3$ &  23.35622528154187 & $x^2$ & -4.027321052314352 \\*
   &     &      &   $x$ &                  0 &   $1$ &                  0 \\
   &0.25 &  0.5 & $x^3$ &  280.3889896276431 & $x^2$ & -286.9065184334287 \\*
   &     &      &   $x$ &  93.24595537566320 &   $1$ & -9.647675950503985 \\
   & 0.5 & 0.75 & $x^3$ &  400.2203835185070 & $x^2$ & -791.8720820637919 \\*
   &     &      &   $x$ &  508.3379735878785 &   $1$ & -105.9312183853788 \\
   &0.75 &    1 & $x^3$ & -1674.366086197562 & $x^2$ &  4222.513185710944 \\*
   &     &      &   $x$ & -3512.375260428360 &   $1$ &  964.2281609149772 \\
\hline
\end{longtable}

\begin{longtable}{ | c | r r |  r |  r | r | r | }
\caption{ Coefficients of the wavelet generators corresponding to the terms $x^3$, $x^2$, $x$, and $1$ on $\left[ a, b \right]$. } 
\label{Tab4} \\

\hline 
 & a & b  &   &  coefficient  &  &  coefficient \\ \hline
\endfirsthead

\multicolumn{6}{c}%
 {{\bfseries \tablename\ \thetable{} -- continued from previous page}} \\
\hline
 & a & b  &   &  coefficient  &  &  coefficient \\ \hline
\endhead

\hline
\endfoot

\hline
\endlastfoot

 \hline                                                              
$\psi_1$&0&0.125&$x^3$ &  512.7201177918754 & $x^2$ & -43.88706601604656 \\*
          &  &  & $x$ &                  0 &   $1$ &                  0 \\
& 0.125 & 0.25 & $x^3$ &  6423.533783310561 & $x^2$ & -3295.876989270898 \\*
            &  & & $x$ &  535.9280902425246 &   $1$ & -27.72322666992469 \\ 
& 0.25 & 0.375 & $x^3$ &  6249.063735441256 & $x^2$ & -6401.764292545093 \\*
            &  & & $x$ &  2121.584875855116 &   $1$ & -227.2933721204776 \\ 
 & 0.375 & 0.5 & $x^3$ & -10261.35309765599 & $x^2$ &  13461.12463292785 \\*
          &  & &   $x$ & -5810.249716786693 &   $1$ &  824.5923576584304 \\ 
 & 0.5 & 0.625 & $x^3$ & -291.9822638070393 & $x^2$ &  622.1814317187671 \\*
            &  & & $x$ & -448.3346409643185 &   $1$ &  107.1992658183961 \\ 
& 0.625 & 0.75 & $x^3$ & -4060.994255437327 & $x^2$ &  8349.124694297900 \\*
            &  & & $x$ & -5690.202791496491 &   $1$ &  1285.198591225143 \\ 
& 0.75 & 0.875 & $x^3$ & -528.1868779043331 & $x^2$ &  1361.218410522837 \\*
            &  & & $x$ & -1169.955815420823 &   $1$ &  335.3075313951340 \\ 
  & 0.875 & 1 & $x^3$ & -130.6441212212436 & $x^2$ &  371.4479197042459 \\*
            &  & & $x$ & -350.9634757447609 &   $1$ &  110.1596772617586 \\ 
\hline
$\psi_2$&0&0.125&$x^3$ &  181.2722522251391 & $x^2$ & -14.15090760709985 \\*
      &     &   &  $x$ &                  0 &   $1$ &                  0 \\
& 0.125 & 0.25 & $x^3$ &  2419.148109670498 & $x^2$ & -1245.470104784877 \\*
      &     &   & $x$ &  202.9293684766930 &   $1$ & -10.49765988775683 \\
& 0.25 & 0.375 & $x^3$ &  2576.675107217909 & $x^2$ & -2596.000954579617 \\*
      &     &   & $x$ &  848.6584813339236 &   $1$ & -89.98311932657152 \\
 & 0.375 & 0.5 & $x^3$ & -8471.127674575379 & $x^2$ &  10735.23758057596 \\*
      &     &   & $x$ & -4488.978621463715 &   $1$ &  619.5243500374206 \\
 & 0.5 & 0.625 & $x^3$ & -3086.030528892990 & $x^2$ &  6003.939215143344 \\*
      &     &   & $x$ & -3796.503115292892 &   $1$ &  782.9740450998639 \\
& 0.625 & 0.75 & $x^3$ &  12658.02295546319 & $x^2$ & -26273.32787303766 \\*
      &     &   & $x$ &  18100.01806795347 &   $1$ & -4137.807295812551 \\
& 0.75 & 0.875 & $x^3$ &  1896.924386457221 & $x^2$ & -4813.188115086536 \\*
      &     &   & $x$ &  4069.162266224350 &   $1$ & -1146.155599563828 \\
  & 0.875 & 1  & $x^3$ &  253.9044233453815 & $x^2$ & -726.4312505262515 \\*
      &     &   & $x$ &  691.1492310163587 &   $1$ & -218.6224038354886 \\
 \hline
$\psi_3$&-1&-0.75&$x^3$ & -12.77748177259837 & $x^2$ & -36.12926000986063 \\*
            &  &  & $x$ & -33.92607470192613 &   $1$ & -10.57429646466388 \\
  &-0.75 &-0.5 & $x^3$ & -153.4635771854685 & $x^2$ & -303.3692884353265 \\*
             &  & & $x$ & -197.3783313309066 &   $1$ & -42.19291944937926 \\ 
  &-0.5 &-0.375 & $x^3$ & -59.47798764121307 & $x^2$ &  2.242882836829032 \\*
             &  & & $x$ &  37.74464778305737 &   $1$ &  10.71372598259577 \\ 
 &-0.375 &-0.25 & $x^3$ &  1758.865605322497 & $x^2$ &  1729.275753199708 \\*
             &  & & $x$ &  565.9055972736512 &   $1$ &  61.79929755698429 \\ 
 &-0.25 &-0.125 & $x^3$ &  3444.276587955715 & $x^2$ &  1712.829251641442 \\*
             &  & & $x$ &  241.6677872507898 &   $1$ &  8.102298002304606 \\ 
    &-0.125 & 0 & $x^3$ & -5400.816676245629 & $x^2$ & -882.1760744370164 \\*
             &  & & $x$ &  7.530202490613302 &   $1$ &  2.106485345615197 \\ 
    & 0 & 0.125 & $x^3$ &  1247.729690468656 & $x^2$ & -454.3272932130656 \\*
             &  & & $x$ &  7.530202490613302 &   $1$ &  2.106485345615197 \\ 
  &0.125 & 0.25 & $x^3$ & -5778.096772948966 & $x^2$ &  3410.751764470566 \\*
             &  & & $x$ & -629.4039464575936 &   $1$ &  35.05371099919686 \\ 
  &0.25 &0.375 & $x^3$ & -558.2316506051182 & $x^2$ &  527.4560525803158 \\*
             &  & & $x$ & -166.4808009519398 &   $1$ &  17.96851407930143 \\ 
    &0.375 &0.5 & $x^3$ &  226.0107269480922 & $x^2$ & -258.2984965273795 \\*
             &  & & $x$ &  91.98285784857104 &   $1$ & -9.815156131403054 \\ 
  & 0.5 & 0.75 & $x^3$ &  35.52867654721048 & $x^2$ & -72.71045185132442 \\*
             &  & & $x$ &  49.25635097317727 &   $1$ & -11.03865756260972 \\ 
  &0.75 &    1 & $x^3$ &  1.382563578715178 & $x^2$ & -3.919859133343993 \\*
             &  & & $x$ &  3.692027530542450 &   $1$ &  -1.15473197591364 \\ 
 \hline
$\psi_4$&-1&-0.75 &$x^3$ &  6.603401663338006 & $x^2$ &  18.67158529263201 \\*
             &  & &  $x$ &  17.53296559525000 &   $1$ &  5.464781965955998 \\
  & -0.75 & -0.5 & $x^3$ &  79.28543622294657 & $x^2$ &  156.7295257413204 \\*
              &  & & $x$ &  101.9689429489431 &   $1$ &  21.79690680867349 \\ 
 & -0.5 & -0.375 & $x^3$ & -82.45595592132698 & $x^2$ & -161.0461033269963 \\*
              &  & & $x$ & -94.50064201116839 &   $1$ & -17.21165242233730 \\ 
 &-0.375 &-0.25  & $x^3$ & -2246.183853666682 & $x^2$ & -2216.018451884065 \\*
              &  & & $x$ & -722.9071965676484 &   $1$ & -77.98646222284469 \\ 
  &-0.25 &-0.125 & $x^3$ & -3800.606529757320 & $x^2$ & -1924.442967268418 \\*
              &  & & $x$ & -285.6652024928302 &   $1$ & -11.18728580653431 \\ 
     &-0.125 & 0 & $x^3$ &  11937.30882839563 & $x^2$ &  2667.744109956829 \\*
              &  & & $x$ &  124.6667844000623 &   $1$ & -0.910594592674766 \\ 
     & 0 & 0.125 & $x^3$ &  3136.965997055351 & $x^2$ & -1389.100592703720 \\*
              &  & & $x$ &  124.6667844000623 &   $1$ & -0.910594592674766 \\ 
  & 0.125 & 0.25 & $x^3$ & -7152.501310626270 & $x^2$ &  4271.385867623253 \\*
              &  & & $x$ & -808.1360506341049 &   $1$ &  47.34127467930285 \\ 
  & 0.25 & 0.375 & $x^3$ & -927.7525957207491 & $x^2$ &  901.3927836102403 \\*
              &  & & $x$ & -290.2798926723839 &   $1$ &  31.24010426928711 \\ 
  & 0.375 & 0.5 & $x^3$ &  154.3100993503755 & $x^2$ & -168.8298334848166 \\*
              &  & & $x$ &  55.89187066577802 &   $1$ & -5.136151388922590 \\ 
    & 0.5 & 0.75 & $x^3$ &  35.12445303753813 & $x^2$ & -71.05734845883623 \\*
              &  & & $x$ &  47.50862037442566 &   $1$ & -10.48944171063683 \\ 
     & 0.75 & 1  & $x^3$ &  1.892186966470384 & $x^2$ & -5.357215161018598 \\*
              &  & & $x$ &  5.037869422626043 &   $1$ & -1.572841228077830 \\ 
 \hline
$\psi_5$&-1&-0.75&$x^3$ & -0.001049843371405 & $x^2$ & -0.003037046895837 \\*
            &  &  & $x$ & -0.002924563677460 &   $1$ & -0.000937360153027 \\
& -0.75 & -0.5 &  $x^3$ & -0.135429794979682 & $x^2$ & -0.283795160078696 \\*
            &  &  & $x$ & -0.197295565112758 &   $1$ & -0.045480714648891 \\ 
& -0.5 & -0.375 & $x^3$ &  64.90438914754031 & $x^2$ &  92.49504460139021 \\*
             &  & & $x$ &  43.80167998946610 &   $1$ &  6.889274490088309 \\ 
&-0.375 & -0.25 & $x^3$ &  887.7169530732866 & $x^2$ &  874.0398882738802 \\*
             &  & & $x$ &  282.8362623376594 &   $1$ &  30.01300552998866 \\ 
& -0.25 &-0.125 & $x^3$ &  689.7712273209479 & $x^2$ &  312.8547139314197 \\*
             &  & & $x$ &  39.35849874499266 &   $1$ &  1.124736063345467 \\ 
  & -0.125 & 0 & $x^3$ & -677.4429634928239 & $x^2$ & -100.9319418261331 \\*
             &  & & $x$ &                  0 &   $1$ &                  0 \\ 
    & 0 & 0.125 & $x^3$ &  222.6292417447635 & $x^2$ &  78.22658535063456 \\*
             &  & & $x$ &                  0 &   $1$ &                  0 \\ 
 & 0.125 & 0.25 & $x^3$ &  13436.75267042187 & $x^2$ & -7191.190145000426 \\*
             &  & & $x$ &  1197.942146868526 &   $1$ & -61.96696676846536 \\ 
 & 0.25 & 0.375 & $x^3$ &  3309.658731654356 & $x^2$ & -3681.684248151221 \\*
             &  & & $x$ &  1342.019311962833 &   $1$ & -159.0945338018750 \\ 
  & 0.375 & 0.5 & $x^3$ &  1623.482149916625 & $x^2$ & -1948.015296157701 \\*
             &  & & $x$ &  753.1233433882986 &   $1$ & -93.13627378293758 \\ 
  & 0.5 & 0.75 & $x^3$ &  334.8296335785243 & $x^2$ & -668.4089674911632 \\*
             &  & & $x$ &  440.0064019753428 &   $1$ & -95.39782070083224 \\ 
     & 0.75 & 1 & $x^3$ &  23.73541224062620 & $x^2$ & -67.14137193635135 \\*
             &  & & $x$ &  63.07650715082408 &   $1$ & -19.67054745509894 \\ 
 \hline
$\psi_6$&-1&-0.75&$x^3$ &  29.39534138606089 & $x^2$ &  83.11643728617156 \\*
            &  &  & $x$ &  78.04685041416045 &   $1$ &  24.32575451404978 \\
 & -0.75 & -0.5 & $x^3$ &  351.2049102801593 & $x^2$ &  694.0258200020245 \\*
            &  & &  $x$ &  451.3572769791491 &   $1$ &  96.43545853732169 \\ 
 &  -0.5 &-0.375& $x^3$ &  1275.659605723797 & $x^2$ &  1615.831201469838 \\*
             &  & & $x$ &  679.8216368642327 &   $1$ &  95.77313004336493 \\ 
 &-0.375 &-0.25 & $x^3$ &  11255.59466789715 & $x^2$ &  11095.38855970805 \\*
             &  & & $x$ &  3579.204551188512 &   $1$ &  376.2646074570184 \\ 
 &-0.25 &-0.125 & $x^3$ &  5327.416704802243 & $x^2$ &  2249.311342113261 \\*
             &  & & $x$ &  267.6993104714105 &   $1$ &  8.640342704059652 \\ 
 & -0.125 &   0 & $x^3$ & -2957.628236410947 & $x^2$ & -374.9318262498545 \\*
             &  & & $x$ &                  0 &   $1$ &                  0 \\ 
 &  0  &  0.125 & $x^3$ & -46.42927388024598 & $x^2$ &  4.175934965984637 \\*
             &  & & $x$ &                  0 &   $1$ &                  0 \\ 
 & 0.125 & 0.25 & $x^3$ & -788.4860165790724 & $x^2$ &  412.3997958564289 \\*
             &  & & $x$ & -67.27205540860357 &   $1$ &  3.479838675245901 \\ 
 & 0.25 & 0.375 & $x^3$ & -231.1352598754021 & $x^2$ &  260.5391696188435 \\*
             &  & & $x$ & -95.84500917174907 &   $1$ &  11.40576068238651 \\ 
 & 0.375 &  0.5 & $x^3$ & -133.3790261589963 & $x^2$ &  159.4675576509780 \\*
             &  & & $x$ & -61.28221129495857 &   $1$ &  7.502793024182585 \\ 
 &  0.5 &  0.75 & $x^3$ & -26.46366034181517 & $x^2$ &  52.95015549352690 \\*
             &  & & $x$ & -34.95133350039340 &   $1$ &  7.602283939115114 \\ 
 &  0.75  &   1 & $x^3$ & -1.798568145773455 & $x^2$ &  5.088295556487342 \\*
             &  & & $x$ & -4.780886675654469 &   $1$ &  1.491159264940578 \\ 
 \hline
$\psi_{L_1}$&0&0.125& $x^3$ & -1425.503021368171 & $x^2$ &  772.9341296224186  \\*
                 &  & & $x$ & -56.55471808022630 &   $1$ &                  0  \\ 
 &  0.125  &  0.25  & $x^3$ &  16500.24187884111 & $x^2$ & -9088.459925693413  \\*
                 &  & & $x$ &  1568.524503551422 &   $1$ & -84.06184109786738  \\ 
 &   0.25  &  0.375 & $x^3$ &  3527.021590160490 & $x^2$ & -3847.552266045169  \\*
                 &  & & $x$ &  1380.549477854916 &   $1$ & -161.9182463911214  \\ 
 &  0.375  &    0.5 & $x^3$ &  1363.434660038630 & $x^2$ & -1645.715891236687  \\*
              &   &  & $x$ &  641.9354328937131 &   $1$ & -80.47581521996845  \\ 
 &     0.5 &   0.75 & $x^3$ &  281.9469114612293 & $x^2$ & -562.0738058307491  \\*
                 &  & & $x$ &  369.4091589208262 &   $1$ & -79.93723101283425  \\ 
 &    0.75 &      1 & $x^3$ &  20.47469809743168 & $x^2$ & -57.91380903680889  \\*
                 &  & & $x$ &  54.40352378132468 &   $1$ & -16.96441284194740  \\ 
            \hline                                             
$\psi_{L_2}$&0&0.125& $x^3$ &  3578.516685827048 & $x^2$ & -1483.144649002284  \\*
                 &  & & $x$ &  125.1644746095171 &   $1$ &                  0  \\ 
 &  0.125  &  0.25  & $x^3$ & -3221.524685507020 & $x^2$ &  2257.721533142182  \\*
                 &  & & $x$ & -491.3001316453148 &   $1$ &  31.88837248923357  \\ 
 & 0.25   & 0.375   & $x^3$ &  253.7811404150473 & $x^2$ & -437.2143373073622  \\*
            &   &   &   $x$ &  204.5479612190696 &   $1$ & -27.94181235379833  \\ 
 &   0.375 &    0.5 & $x^3$ &  859.2458117883292 & $x^2$ & -1011.673202197156  \\*
            &     & &   $x$ &  379.9617016508116 &   $1$ & -44.86748817002474  \\ 
 &   0.5 &   0.75   & $x^3$ &  179.5608184460906 & $x^2$ & -359.6822707980349  \\*
            &  &    &   $x$ &  237.7345152583694 &   $1$ & -51.79100365580411  \\ 
  &  0.75  &   1    & $x^3$ &  11.94553851645503 & $x^2$ & -33.79701898388310  \\*
            &     & &   $x$ &  31.75742241840201 &   $1$ & -9.905941950973904  \\ 
 \hline
$\psi_{R1}$&0&0.25& $x^3$ &  7.391479105551183  & $x^2$ & -1.275014819251857  \\*
             &  &  &  $x$ &                  0  &   $1$ &                  0  \\
& 0.25  & 0.5     & $x^3$ &  87.83823678183982  & $x^2$ & -89.99728641434138  \\*
              &  & & $x$ &  29.27736873324082  &   $1$ & -3.031180797309012  \\ 
&   0.5 &  0.625  & $x^3$ &  780.1233069697320  & $x^2$ & -1281.837786515931 \\*
          &  &   &   $x$ &  701.9040661939105  &   $1$ & -127.9200382757332  \\ 
&  0.625 &  0.75 &  $x^3$ &  8761.559020997937  & $x^2$ & -17645.25953421656  \\*
            &  & &    $x$ &  11802.93627344290  &   $1$ & -2622.696251230958  \\ 
&   0.75 &  0.875 & $x^3$ &  7584.571826637600  & $x^2$ & -19150.38737314525  \\*
              &  & &  $x$ &  16046.79392231900  &   $1$ & -4462.413605869881  \\ 
&  0.875 &    1 &   $x^3$ & -18377.05522246807  & $x^2$ &  51055.94173564137  \\*
              &  & & $x$ & -47183.66988964301  &   $1$ &  14504.78337646970  \\  \hline
$\psi_{R2}$&0&0.25& $x^3$ &  28.49333825597376  & $x^2$ & -4.913915083023494  \\*
              &  &  & $x$ &                  0  &   $1$ &                  0  \\
    & 0.25  & 0.5 & $x^3$ &  340.6257513882141  & $x^2$ & -348.7311999647503  \\*
              &  & & $x$ &  113.3838149785686  &   $1$ & -11.73444239472556  \\ 
  & 0.5 & 0.625  & $x^3$ &  1001.089085127504  & $x^2$ & -1771.030876846587  \\*
              &  & & $x$ &  1040.335991555937  &   $1$ & -202.1935281803620  \\ 
  & 0.625 & 0.75 & $x^3$ &  7909.591723311074  & $x^2$ & -15934.19401753424  \\*
              &  & & $x$ &  10648.38838829413  &   $1$ & -2361.386826210905  \\ 
  & 0.75 & 0.875 & $x^3$ &  1810.118176484035  & $x^2$ & -4899.457987996012  \\*
              &  & & $x$ &  4389.145954257418  &   $1$ & -1300.778614730966  \\ 
  & 0.875 &  1    & $x^3$ &  7611.345386975026  & $x^2$ & -20790.59460977267  \\*
              &  & & $x$ &  18873.94129327007  &   $1$ & -5694.692070472430  \\ 
\end{longtable}

\end{document}